%% file: saddle_extension.tex
\documentclass[12pt,reqno]{amsart}

\usepackage[utf8]{inputenc}

\usepackage{amsmath} 
\usepackage{amsthm} 
\usepackage{amssymb} 
\usepackage{amscd} 

\usepackage{enumitem} 
\usepackage{mathtools} 
\usepackage{mathrsfs} 
\usepackage{hyperref} 
\hypersetup{
    linktocpage = true,
    colorlinks=true, 
    linkcolor=red,
} 
\usepackage{esint} 

\usepackage{graphicx} 

\usepackage{pgf,tikz}
\usetikzlibrary{fadings}

\newtheorem{theorem}{Theorem}[section]
\newtheorem*{theorem*}{Theorem}
\newtheorem{proposition}[theorem]{Proposition}
\newtheorem{lemma}[theorem]{Lemma}
\newtheorem{corollary}[theorem]{Corollary}

\theoremstyle{definition}
\newtheorem{definition}[theorem]{Definition}

\theoremstyle{remark}
\newtheorem{remark}[theorem]{Remark}
\newtheorem*{remark*}{Remark}

\newcommand{\con}[1]{\mathbb{#1}}
\newcommand{\R}{\con{R}} 
\newcommand{\N}{\con{N}} 
\newcommand{\Sph}{\con{S}} 

\newcommand{\ccal}{\mathscr{C}}
\newcommand{\ecal}{\mathcal{E}}
\newcommand{\ical}{\mathcal{I}}

\newcommand{\ncal}{\mathcal{N}}
\newcommand{\ocal}{\mathcal{O}}

\makeatletter
\newcommand{\leqnomode}{\tagsleft@true\let\veqno\@@leqno}
\newcommand{\reqnomode}{\tagsleft@false\let\veqno\@@eqno}
\makeatother

\newcommand{\norm}[1]{\left | \left |{#1} \right | \right |}

\newcommand{\s}{\gamma}
\newcommand{\fraclaplacian}{(-\Delta)^\s}
\newcommand{\Lip}{\mathrm{Lip}}

\renewcommand{\d}{\,\mathrm{d}} 
\newcommand{\dx}{\,\mathrm{d}x} 
\newcommand{\bpar}[1]{\left ( {#1}\right )}
\newcommand{\setcond}[2]{\left \{ #1 \ : \ #2  \right \}}

\newcommand\beqc[1]{\left\{\begin{array}{#1}}
\newcommand\eeqc{\end{array} \right.}
\def\PDEsystem{rcll}
\def\bmatrix{\begin{pmatrix}}
\def\ematrix{\end{pmatrix}}

\DeclareMathOperator{\dist}{dist}

\DeclareMathOperator{\PV}{P.V.}
\let\div\relax
\DeclareMathOperator{\div}{div}

\def\ds{\displaystyle}

\numberwithin{equation}{section}

\graphicspath {{Images/}}

\title[Saddle solution to the fractional Allen-Cahn equation]{Uniqueness and stability of the saddle-shaped solution to the fractional Allen-Cahn equation}
\author{Juan-Carlos Felipe-Navarro} 
\address{J.C. Felipe-Navarro:
Universitat Polit\`ecnica de Catalunya and BGSMath, Departament de Matem\`{a}tiques, Diagonal 647, 08028 Barcelona, Spain}
\email{juan.carlos.felipe@upc.edu}

\author{Tomás Sanz-Perela}
\address{T. Sanz-Perela:
Universitat Polit\`ecnica de Catalunya and BGSMath, Departament de Matem\`{a}tiques, Diagonal 647, 08028 Barcelona, Spain}
\email{tomas.sanz@upc.edu}

\thanks{Both authors acknowledge financial support from the Spanish Ministry of Economy and Competitiveness (MINECO), through the María de Maeztu Programme for Units of Excellence in R\&D (MDM-2014-0445-16-4 and MDM-2014-0445, respectively), are supported by MINECO grants MTM2014-52402-C3-1-P and MTM2017-84214-C2-1-P, are members of the Barcelona Graduate School of Mathematics (BGSMath), and are part of the Catalan research group 2017 SGR 01392.}


\begin{document}

\begin{abstract}
In this paper we prove the uniqueness of the saddle-shaped solution $u\colon \R^{2m} \to \R$ to the semilinear nonlocal elliptic equation $\fraclaplacian u = f(u)$ in $\R^{2m}$, where $\s \in (0,1)$ and $f$ is of Allen-Cahn type. Moreover, we prove that this solution is stable whenever $2m\geq 14$. As a consequence of this result and the connection of the problem with nonlocal minimal surfaces, we show that the Simons cone $\setcond{(x', x'') \in \R^{m}\times \R^m}{|x'| = |x''|}$ is a stable nonlocal $(2\s)$-minimal surface in dimensions $2m\geq 14$.

Saddle-shaped solutions of the fractional Allen-Cahn equation are doubly radial, odd with respect to the Simons cone, and vanish only in this set. It was known that these solutions exist in all even dimensions and are unstable in dimensions $2$, $4$, and $6$. Thus, after our result, the stability remains an open problem only in dimensions $8$, $10$, and $12$.

The importance of studying this type of solution is due to its relation with the fractional version of a conjecture by De Giorgi. Saddle-shaped solutions are the simplest non 1D candidates to be global minimizers in high dimensions, a property not yet established in any dimension.
\end{abstract}

\maketitle

\section{Introduction}

This paper is devoted to the study of saddle-shaped solutions to the fractional Allen-Cahn equation
\begin{equation}
\label{Eq:AllenCahn}
\fraclaplacian u = f(u)  \quad \text{ in } \R^{n}\,,
\end{equation}
where $n=2m$ is an even integer, $f$ is of bistable type (see \eqref{Eq:fHypotheses} below), and $\fraclaplacian$ is the fractional Laplacian, defined for $\s\in(0,1)$ by
$$
\fraclaplacian u (x):= c_{n,\s}  \PV \int_{\R^{n}} \dfrac{u(x) - u(\tilde{x})}{|x-\tilde{x}|^{n+2\s}} \d \tilde{x}\,.
$$
Here $c_{n,\s}>0$ is a normalizing constant depending only on $n$ and $\s$, and $\PV$ stands for principal value. This problem is motivated by the fractional De Giorgi conjecture and it is closely related to the theory of nonlocal minimal surfaces, as we will explain later in this introduction.

Throughout the paper we assume that $f\in C^{2,\alpha}\big((-1,1)\big)$, for some $\alpha\in (0,1)$, and that is of bistable type, i.e.,
\begin{equation}\label{Eq:fHypotheses}
f \text{ is odd, }\; f(0)=f(1)=0,\text{ and }
f''<0 \text{ in } (0,1).
\end{equation}
Note that as a consequence we have $f>0$ in $(0,1)$. A typical example of this kind of nonlinearity is $f(u)=u-u^3$.

An important role in this paper is played by the Simons cone, which is defined in $\R^{2m}$ by
$$
\mathscr{C} := \setcond{x = (x', x'') \in \R^{m}\times \R^m}{|x'| = |x''|}\,.
$$
It is well known that the Simons cone has zero mean curvature at every point $x \in \ccal \setminus \{0\}$, in every dimension $2m \geq 2$. However, it is only in dimensions $2m \geq 8$ that $\ccal$ is a minimizer of the area functional, as established by Bombieri, De Giorgi, and Giusti in \cite{BombieriDeGiorgiGiusti}. Regarding the fractional setting, for every $\s\in(0,1/2)$, $\ccal$ has zero \emph{nonlocal} mean curvature in every even dimension but it is not known if, in addition,  it is a minimizer of the fractional perimeter in dimensions $2m\geq 8$. We recall that it is only in dimension $2m=2$ where we have a complete classification of minimizing nonlocal minimal cones, establishing that they must be flat (see \cite{SavinValdinoci-Cones}). The same classification result for \emph{stable} nonlocal minimal cones holds also in $\R^2$ (see \cite{SavinValdinoci-Stable}), and in $\R^3$ and for $\s$ close to $1/2$ (see \cite{CabreCintiSerra-Cones}). Recall that by stability we understand that the second variation of the energy functional is nonnegative (and thus, it is a weaker property than minimality). In higher dimensions $n$, the classification of nonlocal minimal cones is widely open and the main result in this direction is the one in \cite{CaffarelliValdinoci}, establishing that minimizing nonlocal minimal cones are flat in dimensions $2\leq n\leq 7$ for $\gamma$ close to $1/2$.  It is also an open problem to find, in high dimensions, an example of nonsmooth minimizing nonlocal minimal surface. A main candidate for this is, as in the local case, the Simons cone.

The only other result (apart from the previous ones) concerning the possible minimality of the Simons cone refers to its stability, and it is proved in \cite{DaviladelPinoWei} by D\'avila, del Pino, and Wei.  In that paper, the authors characterize the stability of Lawson cones through an inequality involving only two hypergeometric constants which depend only on $\s$ and the dimension $n$. It is a hard task to verify the criterion analytically, and this has not been accomplished. It seems also delicate to check it numerically, but some cases are treated in \cite{DaviladelPinoWei}. With a numerical computation, \cite{DaviladelPinoWei} finds that, in dimensions $n \leq 6$ and for $\s $ close to zero, no Lawson cone with zero nonlocal mean curvature is stable. The Simons cone is a particular case of Lawson cone corresponding to $C_m^m(2\s)$ in the notation of \cite{DaviladelPinoWei}. Numerics also shows that all Lawson cones in dimension $7$ are stable if $\s$ is close to zero. These results for small $\s$ fit with the general belief that, in the fractional setting, the Simons cone should be stable (and even a minimizer) in dimensions $2m \geq 8$ (as in the local case), probably for all $\s\in(0,1/2)$, though this is still an open problem. 

In the present paper, we make a first contribution to the previous question by showing that the Simons cone is a stable $(2\s)$-minimal cone in dimensions $2m\geq 14$. Our proof uses the so-called saddle-shaped solution to the Allen-Cahn equation. As we will see in more detail, by the fractional Modica-Mortola type $\Gamma$-convergence result, the remarks above on the stability of the Simons cone are expected to hold also for saddle-shaped solutions. Indeed, our proof proceeds by establishing the stability of such solution to the fractional Allen-Cahn equation in dimensions $2m \geq 14$ (see Theorem~\ref{Thm:Stability} below). Then, as a consequence of this and a recent result by Cabr\'e, Cinti, and Serra in \cite{CabreCintiSerra-Stable} (see also the comments in \cite{CabreCintiSerra-Cones}) concerning the preservation of stability along a blow-down procedure for the fractional Allen-Cahn equation, we deduce the stability of the Simons cone as a nonlocal minimal surface in these dimensions (see Corollary~\ref{Cor:SimonsConeStableDim14}). 

To introduce saddle-shaped solutions, we define the following variables:
$$
s := \sqrt{x_1^2 + \ldots + x_m^2 } \quad \text{ and } \quad 
t := \sqrt{x_{m+1}^2 + \ldots + x_{2m}^2}\,,
$$
for which the Simons cone becomes $\mathscr{C} = \{s=t\}$. Through the paper we will also use the letter $\ocal$ to denote one of the sets in which the cone divides the space:
$$
\ocal:= \setcond{x = (x', x'') \in \R^{m}\times \R^m}{|x'| > |x''|} = \{s > t\}.
$$

We define saddle-shaped solutions as follows.

\begin{definition}
	\label{Def:SaddleShapedSol}
	We say that a bounded solution $u$ to \eqref{Eq:AllenCahn} is a \emph{saddle-shaped solution} (or simply \emph{saddle solution}) if
	\begin{enumerate}[label=(\roman{*})]
		\item $u$ is a doubly radial function, that is, $u = u(s,t)$.
		\item $u$ is odd with respect to the Simons cone, that is, $u(s,t)=-u(t,s)$.
		\item $u > 0$ in $\ocal = \{s>t\}$.
	\end{enumerate}
\end{definition}

Saddle-shaped solutions for the classical Allen-Cahn equation involving the Laplacian were first studied by Dang, Fife, and Peletier in \cite{DangFifePeletier} in dimension $2m=2$. They established the existence and uniqueness of this type of solutions, as well as some monotonicity properties and asymptotic behavior. In \cite{Schatzman}, Schatzman studied the instability property of saddle solutions in $\R^2$. Later, Cabr\'e and Terra  proved the existence of a saddle solution in every dimension $2m\geq 2$, and they established some qualitative properties such as asymptotic behavior, monotonicity properties, as well as instability in dimensions $2m = 4$ and $2m = 6$ (see \cite{CabreTerraI,CabreTerraII}). The uniqueness in dimensions higher than $2$ was established by Cabr\'e in \cite{Cabre-Saddle}, where he also proved that the saddle solution is stable in dimensions $2m \geq 14$.

In the nonlocal framework, there are only two works concerning saddle-shaped solutions to \eqref{Eq:AllenCahn}. In  \cite{Cinti-Saddle,Cinti-Saddle2}, first for $\s=1/2$ and then for $\s\in(0,1)$, Cinti proved the existence of a saddle-shaped solution to \eqref{Eq:AllenCahn} as well as some qualitative properties such as asymptotic behavior, monotonicity properties, and instability in low dimensions (see Theorem~\ref{Th:Summary} below). 

In the present paper, we prove further properties of these solutions, the main ones being uniqueness and, when $2m\geq 14$, stability. Uniqueness is important since then the saddle-shaped solution becomes a canonical object associated to the Allen-Cahn equation and the Simons cone. 

In \cite{Cinti-Saddle,Cinti-Saddle2}, the main tool used is the extension problem for the fractional Laplacian (see \eqref{Eq:AllenCahnWithExtension} below and \cite{CaffarelliSilvestre}). This is also the approach of the present paper. It should be remarked that the extension technique has the limitation that it only works for the fractional Laplacian, and therefore the same arguments cannot be carried out for more general integro-differential operators of the form
$$
L_K u(x) = \PV \int_{\R^n} \{u(x) - u(\tilde{x})\} K(x-\tilde{x})\d \tilde{x}.
$$
In two forthcoming papers \cite{FelipeSanz-Perela:IntegroDifferentialI,FelipeSanz-Perela:IntegroDifferentialII} we address this problem by studying saddle-shaped solutions to equation $L_K u = f(u)$ in $\R^{2m}$, where $L_K$ is an elliptic integro-differential operator of the previous form with a radially symmetric kernel $K$. One of the most basic tools that we need is a maximum principle for the operator acting on  functions which are odd with respect to the Simons cone. In \cite{FelipeSanz-Perela:IntegroDifferentialI} we find a necessary and sufficient condition to have such a maximum principle and, as we will see there, this will require a certain convexity property of the kernel $K$.

Let us now introduce the extension problem for the fractional Laplacian, which is the main tool used in this paper. First we should settle the notation. We call $\R^{n+1}_+ := \R^n \times (0, +\infty)$ and denote points by $(x,\lambda)\in \R^{n+1}_+$ with $x\in \R^n$ and $\lambda > 0$. As it is well known, see~\cite{CaffarelliSilvestre}, if $u:\R^{n+1}_+ \to \R$ solves $\div(\lambda^a \nabla u) = 0$ in $\R^{n+1}_+$ with $a=1-2\s$, then
$$
\dfrac{\partial u}{\partial \nu^a} (x) := -\lim_{\lambda \downarrow 0} \lambda^a u_\lambda (x, \lambda) = \dfrac{\fraclaplacian u (x,0)}{d_\s},
$$
where $d_\s$ is a positive constant depending only on $\s$. Therefore, problem \eqref{Eq:AllenCahn} is equivalent to
\begin{equation}
\label{Eq:AllenCahnWithExtension}
\beqc{\PDEsystem}
\div(\lambda^a \nabla u) &=& 0 & \textrm{ in } \R^{n+1}_+\,, \\
d_\s \dfrac{\partial u}{\partial \nu^a} &=& f(u) & \textrm{ on } \partial \R^{n+1}_+ = \R^n\,.
\eeqc
\end{equation}

We will always consider functions defined in $\R^{n+1}_+$ and not only in $\R^n$, and we will use the same letter to denote both the function and its trace on $\R^n$. Regarding sets in $\R^{n+1}_+$, we use the following notation. If $\Omega \subset \R^{n+1}_+$, we define
\begin{equation}
\label{Eq:DefBoundaries}
\partial_L \Omega := \overline{\partial \Omega \cap \{\lambda > 0\}}
\quad \text{ and } \quad \partial_0 \Omega := \partial \Omega \setminus \partial_L \Omega \subset \{\lambda = 0\}\,.
\end{equation}
We write
$$
B_{R}^+ := \setcond{(x,\lambda)\in \R^{n+1}_+}{|(x,\lambda)|< R},
$$
for half-balls in $\R^{n+1}_+$. If $x_0\in \R^n$, $B_{R}^+(x_0)=(x_0,0)+B_{R}^+$.

A certain solution of problem~\eqref{Eq:AllenCahn} in dimension 1, the so-called \emph{layer solution}, plays a crucial role through this paper. It is the unique solution of the following problem:

\begin{equation}
\label{Eq:LayerSolution}
\beqc{\PDEsystem}
\div(\lambda^a \nabla u_0) &=& 0 & \textrm{ in } \R^{2}_+ = \R \times (0, +\infty)\,, \\
d_\s \dfrac{\partial u_0}{\partial \nu^a} &=& f(u_0) & \textrm{ on } \partial \R^{2}_+ = \R\,,\\
\partial_x u_0 &>& 0 & \textrm{ on } \partial \R^{2}_+ = \R\,,\\
u_0 (0,0) &=& 0 \,,\\
\ds \lim_{x\to\pm \infty} u_0 (x,0) &=& \pm 1\,.
\eeqc
\end{equation}
Under the assumptions on $f$ in \eqref{Eq:fHypotheses}, the existence and uniqueness of such solution are well known (see \cite{CabreSireI}).

The importance of the layer solution comes from the fact that the associated function
\begin{equation}
\label{Eq:DefULayer}
U(x,\lambda) := u_0\left( \frac{s-t}{\sqrt{2}},\lambda \right) \ \ \ \text{ for } x\in\R^{2m} \text{ and } \lambda>0,
\end{equation}
which is odd with respect to the Simons cone and positive in $\ocal\times [0,+\infty)$, describes the asymptotic behavior of saddle-shaped solutions at infinity (as shown in \cite{Cinti-Saddle, Cinti-Saddle2}; see Theorem~\ref{Th:Summary} below).  Note that from Lemma~4.2 in \cite{CabreTerraI}, we know that $|s-t|/\sqrt{2}$ is the distance to the Simons cone. Therefore, we can understand the function $U$ as the layer solution centered at each point of the Simons cone and oriented in the normal direction to the cone. Moreover, in this paper we  show (see Proposition~\ref{Prop:SaddleUnderLayer}) that the saddle-shaped solution lies below $U$ in $\ocal$, as it occurs in the local case (see Proposition 1.5 in \cite{CabreTerraI}).

It is sometimes useful to consider also the following variables:
$$
y := \dfrac{s+t}{\sqrt{2}} \quad \text{ and } \quad z := \dfrac{s-t}{\sqrt{2}} \,,
$$
which satisfy $y\geq 0$ and $-y \leq z \leq y$. In these variables, $\ccal = \{ z = 0\}$ and $\ocal = \{ z > 0\}$. Therefore, we can write $U(x,\lambda) = u_0(z, \lambda)$.

To study the minimality and stability of the saddle-shaped solution, we recall the energy functional associated to equation \eqref{Eq:AllenCahnWithExtension}:
$$
\ecal (w, \Omega) := \dfrac{d_\s}{2}\int_{\Omega} \lambda^a | \nabla w |^2 \d x \d \lambda + \int_{\partial_0 \Omega} G(w)\d x\,, \quad \text{ where }\ G' = -f\,.
$$
We say that $u$ is a minimizer for problem \eqref{Eq:AllenCahnWithExtension} in  $\Omega \subset \R^{2m+1}_+$ if
$$
\ecal(u,\Omega) \leq \ecal(w,\Omega)
$$
for every $w$ such that $w=u$ on $\partial_L \Omega$. Observe that the admissible competitors do not have the boundary condition prescribed on $\partial_0\Omega$. This is in correspondence with the Neumann condition in \eqref{Eq:AllenCahnWithExtension}. We say that $u$ is a global minimizer if it is a minimizer in every bounded domain $\Omega$ of $\R^{2m+1}_+$.

A bounded solution to \eqref{Eq:AllenCahnWithExtension} is said to be \emph{stable} if the second variation of the energy with respect to perturbations $\xi$ which have compact support in $\overline{\R^{2m+1}_+}$ is nonnegative. That is, if
\begin{equation}
\label{Eq:StabilityCondition}
\int_{\R^{2m}} f'(u) \, \xi^2 \d x  \leq d_\s \int_0^\infty \int_{\R^{2m}} \lambda^a \, |\nabla \xi|^2 \d x \d \lambda 
\end{equation}
for every $\xi \in C^\infty_c(\overline{\R^{2m+1}_+})$.

In the following theorem we collect the known results concerning saddle-shaped solutions to \eqref{Eq:AllenCahn}. 

\begin{theorem}[\cite{Cinti-Saddle,Cinti-Saddle2,CabreSolaMorales,CabreSireII}]
	\label{Th:Summary}
	Let $\s \in (0,1)$  and let $f\in C^{2,\alpha}\big((-1,1)\big)$ be a function satisfying \eqref{Eq:fHypotheses}.
	\begin{enumerate}[label=(\roman{*})]
		\item For every even dimension $2m\geq 2$, there exists a saddle-shaped solution to  problem~\eqref{Eq:AllenCahn} with $|u|<1$.
		\item For every even dimension $2m\geq 2$, every saddle-shaped solution to problem~\eqref{Eq:AllenCahn} satisfies
		$$ \big|\big| \, |u-U| + |\nabla_x(u-U)| \, \big|\big|_{L^\infty(\R^{2m}\setminus B_R)} \to 0, \ \ \ \text{ as } \ R\to+\infty, $$
		where $U$ is defined in \eqref{Eq:DefULayer}.
		\item In  dimension $2m$ with $2\leq 2m \leq 6$, every saddle-shaped solution is unstable.
	\end{enumerate}
\end{theorem}

Here $\nabla_x$ denotes the gradient only in the horizontal variables $x\in \R^{2m}$, not to be confused with the gradient $\nabla = \nabla_{(x,\lambda)}$ in \eqref{Eq:AllenCahnWithExtension} or \eqref{Eq:StabilityCondition}, for instance.

Points (i) and (ii) of Theorem~\ref{Th:Summary} were proved by Cinti, first for $\s = 1/2$  in \cite{Cinti-Saddle} and then extended to all powers $\s \in (0,1)$ in \cite{Cinti-Saddle2}. Instability in dimension $2m = 2$ follows from a general result on stable solutions established in \cite{CabreSireII} (previously proved for $\s = 1/2$ in \cite{CabreSolaMorales}). Instead, instability in dimensions $2m=4$ and $2m=6$ was proved in \cite{Cinti-Saddle,Cinti-Saddle2}.

Our first main result is the uniqueness of the saddle-shaped solution. As a consequence, such solution to the fractional Allen-Cahn equation becomes a canonical object associated to the cone $\ccal$.

\begin{theorem}
	\label{Thm:Uniqueness}
	Let $\s \in (0,1)$  and let $f$ be a function satisfying \eqref{Eq:fHypotheses}. Then, for every even dimension $2m\geq 2$, there exists a unique saddle-shaped solution to problem~\eqref{Eq:AllenCahnWithExtension}.
\end{theorem}

As in the paper of Cabr\'e \cite{Cabre-Saddle} for the classical case, the proof of the uniqueness result follows from the asymptotic behavior of the saddle solution (point (ii) in Theorem~\ref{Th:Summary}) and a maximum principle in $\ocal$ for the linearized operator at a saddle-shaped solution. The maximum principle is the following.

\begin{proposition}
	\label{Prop:MaxPrincipleLinearizedOperator}
	Let $u$ be a saddle-shaped solution of \eqref{Eq:AllenCahnWithExtension}. 
	Let $\Omega \subset \ocal \times (0,+\infty) \subset \R^{2m+1}_+$ be an open set such that $\partial_0 \Omega$ is nonempty. Let $v \in C^2 (\Omega)\cap C(\overline{\Omega})$ be bounded from above and such that $\lambda^a v_\lambda \in C(\overline{\Omega})$. 
	
	Consider the operator $\mathscr{L}_u $ defined by 
	\begin{equation}
	\label{Eq:LinearizedOperator}
	\mathscr{L}_u v := d_\s \dfrac{\partial v}{\partial \nu^a}  -f'(u) v \ \ \text{ on } \ \partial_0 \Omega \subset \R^{2m}\times\{0\},
	\end{equation}
	and assume that
	$$
	\beqc{\PDEsystem}
	-\div(\lambda^a \nabla v) &\leq& b(x,\lambda) v & \text{ in } \Omega \subset \ocal \times (0,+\infty)\,, \\
	\mathscr{L}_u v &\leq & 0 & \text{ on } \partial_0 \Omega \subset \ocal \,, \\
	v & \leq & 0 & \text{ on } \partial_L \Omega \,,\\
	\ds \limsup_{x\in \partial_0 \Omega, \  |x|\to +\infty} v(x,0) & \leq & 0\,,
	\eeqc
	$$
	with $b \leq 0$. Then, $v\leq 0$ in $\Omega$.
\end{proposition}

To establish the previous maximum principle we follow the proof of the analogous result for the local case ($\s = 1$) in \cite{Cabre-Saddle}. It involves a maximum principle in ``narrow'' sets (see also \cite{Cabre-Topics,BerestyckiNirembergVaradhan}). The main difference between our proof and the one in \cite{Cabre-Saddle} is that, since we are using the extension problem, a new notion of narrowness is needed to carry out the same type of arguments (see Section~\ref{Sec:MaximumPrinciple} for the details).

The second main result of this paper is the following pointwise estimate for the saddle-shaped solution. We prove that the function $U(s,t,\lambda) := u_0 ( (s-t)/\sqrt{2}, \lambda)$ is a barrier for the saddle-shaped solution. This result was established in the local setting ($\s = 1$) in \cite{CabreTerraI}, but in such case the proof is quite simple by using the so-called Modica estimate (see \cite{CabreTerraI} for the details). In the fractional framework, this estimate is only available (in a nonlocal form) in dimension 1 (see \cite{CabreSolaMorales, CabreSireI}) and therefore we need another type of argument. Our strategy is to use a maximum principle for the linearized operator at $U$, similar to the one in Proposition~\ref{Prop:MaxPrincipleLinearizedOperator}. The pointwise estimate we establish is the following.

\begin{proposition}
	\label{Prop:SaddleUnderLayer}
	Let $u$ be the saddle-shaped solution of \eqref{Eq:AllenCahnWithExtension},  let $u_0$ be the layer solution given by \eqref{Eq:LayerSolution} and let $U$ be defined by \eqref{Eq:DefULayer}. Then, 
	\begin{equation}
	\label{Eq:SaddleUnderLayer}
	|u(x,\lambda)| \leq |U(x,\lambda)| = |u_0 ( \dist(x, \ccal), \lambda) | \quad \text{ for every } (x,\lambda)\in \overline{\R^{2m+1}_+}\,.
	\end{equation}
\end{proposition}

The third main result of the present paper establishes the stability of the saddle solution in high dimensions. This is an extension of Theorem~1.4 in \cite{Cabre-Saddle} to the nonlocal case. For its proof, it is crucial to use the extension problem.

\begin{theorem}
	\label{Thm:Stability}
	Assume that $f$ satisfies \eqref{Eq:fHypotheses}. If $2m\geq 14$, then the saddle-shaped solution $u$ of \eqref{Eq:AllenCahnWithExtension} is stable in $\R^{2m+1}_+$, i.e., \eqref{Eq:StabilityCondition} holds. 
	
	Its stability is a consequence of the following fact. For every constant $b>0$ satisfying $b(b-m+2)\leq -(m-1)$, the function
	$$
	\varphi := t^{-b} \, u_s - s^{-b} u_t, 
	$$
	defined in $\R^{2m+1}_+\setminus\{st=0\}$, is even with respect to the Simons cone and is a positive supersolution of the linearized operator. More precisely, $-\div(\lambda^a \nabla \varphi) \geq 0$ in $\R^{2m+1}_+\setminus\{st=0\}$ and $\mathscr{L}_u \varphi \geq 0$ in $\R^{2m}\setminus\{st=0\}$, where $\mathscr{L}_u$ is defined in \eqref{Eq:LinearizedOperator}.
\end{theorem}

An important consequence of this result is Corollary~\ref{Cor:SimonsConeStableDim14}, stated next, on the stability of the Simons cone as a $(2\s)$-minimal surface in dimensions $2m\geq 14$. This is the first analytical proof of its stability for some $\s$ and $m$. It follows directly from the convergence results proved in \cite{CabreCintiSerra-Stable} for stable solutions to the Allen-Cahn equation after a blow-down, together with the preservation of the stability along this procedure (see also the comments at the end of this introduction).

\begin{corollary}
	\label{Cor:SimonsConeStableDim14}
	Let $\s \in (0,1/2)$ and $2m\geq 14$. Then, the Simons cone $\ccal \subset \R^{2m}$ is a stable $(2\s)$-minimal surface.
\end{corollary}

The key ingredients to prove Theorem~\ref{Thm:Stability} are some monotonicity and second derivative properties for the saddle-shaped solution. In fact, $\varphi$ being a positive supersolution will follow from such properties. More precisely, our arguments will use the following.

\begin{proposition}
	\label{Prop:MonotonicityProperties}
	Let $u$ be the saddle-shaped solution to \eqref{Eq:AllenCahnWithExtension}. Then,
	\begin{enumerate}[label=(\roman{*})]
		\item $u_y > 0$ in $\ocal \times [0, +\infty)$\,.
		\item $-u_t > 0$ in $(\ocal \setminus \{ t= 0\}) \times [0,+\infty)$.
		\item $u_{st} > 0$ in $(\ocal\setminus \{ t = 0\})\times [0,+\infty)$.
	\end{enumerate}
	As a consequence, for every direction $\partial_\eta = \alpha \partial_y - \beta \partial_t$, where $\alpha$ and $\beta$ are nonnegative constants, $\partial_\eta u > 0$ in $ \{s > t > 0,\ \lambda \geq 0\}$.
\end{proposition}

The monotonicity properties (i) and (ii) were proved in the papers of Cinti \cite{Cinti-Saddle,Cinti-Saddle2} for the so-called maximal saddle solution ---note that in those papers the uniqueness of the saddle-shaped solution was not known yet. From her result and our uniqueness theorem, (i) and (ii) in Proposition~\ref{Prop:MonotonicityProperties} follow. Nevertheless, we present here a new proof of them by applying the maximum principle for the linearized operator to certain equations satisfied by $u_s$ and $u_t$. A similar argument will establish the new property (iii) for the crossed second derivative $u_{st}$. 

To conclude this introduction, let us comment briefly on the importance of problem \eqref{Eq:AllenCahn} and its relation with a conjecture of De Giorgi and the theory of minimal surfaces.

The interest on problem~\eqref{Eq:AllenCahn} originates from a famous conjecture of De Giorgi for the classical Allen-Cahn equation. It reads as follows. Let $u$ be a bounded solution to $-\Delta  u = u - u^3 $ in $\R^n$ which is monotone in one direction, say $\partial_{x_n} u > 0$. Then, if $n\leq 8$, $u$ is one dimensional, i.e., $u$ depends only on one Euclidean variable. This conjecture was proved to be true in dimension $n=2$ by Ghoussoub and Gui \cite{GhoussoubGui}, and in dimension $n=3$ by Ambrosio and Cabr\'e \cite{AmbrosioCabre}. For dimensions $4\leq n \leq 8$, and under the additional assumption 
\begin{equation}
\label{Eq:SavinCondition}
\lim_{x_n \to \pm \infty} u(x',x_n) = \pm 1 \quad \text{ for all } x'\in \R^{n-1}\,,
\end{equation}
the conjecture was established by Savin \cite{Savin-DeGiorgi} (see also the previous work of Ghoussoub and Gui \cite{GhoussoubGui-odd} in dimensions 4 and 5 for antisymmetric solutions). A counterexample to the conjecture in dimensions $n \geq 9$ was given by del~Pino, Kowalczyk and Wei \cite{delPinoKowalczykWei}. 

The corresponding conjecture in the nonlocal setting, where one replaces the operator $-\Delta$ by $\fraclaplacian$, has been widely studied in the last years. In this framework, the conjecture has been proven to be true in dimension $n=2$ by Cabr\'e and Sol\`a-Morales in \cite{CabreSolaMorales} for $\s=1/2$, and extended to every power $0<\s<1$ by Cabr\'e and Sire in \cite{CabreSireII} and also by Sire and Valdinoci in \cite{SireValdinoci}. In dimension $n=3$, the conjecture has been proved by Cabr\'e and Cinti for $1/2 \leq \s < 1$ in \cite{CabreCinti-EnergyHalfL, CabreCinti-SharpEnergy} and by Dipierro, Farina, and Valdinoci for $0<\s<1/2$ in \cite{DipierroFarinaValdinoci}. Recently, in \cite{Savin-Fractional,Savin-Fractional2} Savin has established the validity of the conjecture in dimensions $4\leq n \leq 8$ and for $1/2 \leq \s < 1$, but assuming the additional hypothesis \eqref{Eq:SavinCondition}. Under the same extra assumption, the conjecture is true in the same dimensions for $0<\s<1/2$ and $\s$ close to $1/2$, as proved by Dipierro, Serra, and Valdinoci in \cite{DipierroSerraValdinoci}. The most recent result concerning the proof of the conjecture is the one by Figalli and Serra in \cite{FigalliSerra}, where they have established the conjecture in dimension $n=4$ and $\s=1/2$ without requiring the additional limiting assumption \eqref{Eq:SavinCondition}. Note that, without \eqref{Eq:SavinCondition}, the analogous result for the Laplacian in dimension $n=4$ is not known. In the forthcoming paper \cite{CabreCintiSerra-Stable}, Cabr\'e, Cinti, and Serra prove the conjecture in dimension $n=4$ for $0<\s<1/2$ and $\s$ sufficiently close to $1/2$. A counterexample to the De Giorgi conjecture for fractional Allen-Cahn equation in dimensions $n \geq 9$ for $\s \in ( 1/2 , 1)$ has been very recently announced in \cite{ChanLiuWei}.

Coming back to the local Allen-Cahn equation, while studying this conjecture by De Giorgi, another question arose naturally: do global minimizers in $\R^n$ of the Allen-Cahn energy have one-dimensional symmetry? A deep result from Savin \cite{Savin-DeGiorgi} states that in dimension $n \leq 7$ this is indeed true. On the other hand, Liu, Wang, and Wei \cite{LiuWangWei} have constructed minimizers in dimensions $n\geq 8$ which are not one-dimensional. We should mention that the same question for stable solutions (instead of minimizers) is still largely open, only solved in dimension $n=2$ (see \cite{GhoussoubGui,BerestyckiCaffarelliNiremberg-Qualitative}).

The saddle-shaped solution is of special interest regarding the previous two questions. It is expected to be a simple example of non one-dinsional minimizer to the Allen-Cahn equation in high dimensions, having the same role as the Simons cone for the theory of minimal surfaces. In addition, regarding the conjecture by De Giorgi, if the saddle-shaped solution was proved to be a minimizer in some even dimension $2m$, we would automatically have a counterexample to the conjecture in higher dimensions. This is due to a result by Jerison and Monneau \cite{JerisonMonneau}, where they show that such a counterexample in dimension $\R^{n+1}$ can be constructed with a rather natural procedure if there exists a non one-dimensional global minimizer of $-\Delta u = f(u)$ in $\R^n$ which is bounded and even with respect to each coordinate. The saddle-shaped solution is of special interest in relation with the Jerison-Monneau program since it is even with respect to all the coordinate axis and it is expected to be a minimizer in high dimensions. If proved to be a minimizer, the saddle-shaped solution would provide an alternative construction of a counterexample to the original conjecture of De Giorgi, different from the one of \cite{delPinoKowalczykWei}.

Let us explain why the Allen-Cahn equation has a very strong connection with the theory of minimal surfaces. A deep result from the seventies by Modica and Mortola (see \cite{Modica,ModicaMortola}) states that considering an appropriately rescaled version of the Allen-Cahn equation, the corresponding energy functionals $\Gamma$-converge to the perimeter functional. Thus, the minimizers of the equation converge to the characteristic function of a set of minimal perimeter. This same fact holds for the equation with the fractional Laplacian, though we have two different scenarios depending on the parameter $\s \in (0,1)$. If $\s \geq 1/2$, the rescaled energy functionals associated to \eqref{Eq:AllenCahn} $\Gamma$-converge to the classical perimeter (see \cite{AlbertiBouchitteSeppecher,Gonzalez}), while in the case $\s \in (0,1/2)$ they $\Gamma$-converge to the fractional perimeter (see \cite{SavinValdinoci-GammaConvergence}). As a consequence, if the saddle-shaped solution was proved to be a minimizer in a certain dimension for some $\s \in (0,1/2)$, it would follow that the Simons cone $\ccal$ would be a minimizing nonlocal $(2\s)$-minimal surface in such dimensions. As mentioned before, this last statement is an open problem in any dimension. Our Corollary~\ref{Cor:SimonsConeStableDim14} on stability is related to this question, but for a weaker property than minimality.

By a result of Cabr\'e, Cinti, and Serra in \cite{CabreCintiSerra-Stable}, also the stability is preserved in the blow-down limit when $\s\in(0,1/2)$. Therefore, a limit of stable solutions to \eqref{Eq:AllenCahn} with $\s \in (0,1/2)$ will be a stable set for the $(2\s)$-perimeter. Thus, as a consequence of Theorem~\ref{Thm:Stability} we deduce Corollary~\ref{Cor:SimonsConeStableDim14}.

The paper is organized as follows. In section \ref{Sec:MaximumPrinciple} we prove the maximum principle for the linearized operator in $\ocal$, Proposition~\ref{Prop:MaxPrincipleLinearizedOperator}. Section \ref{Sec:Uniqueness} is devoted to show Theorem~\ref{Thm:Uniqueness} concerning the uniqueness of the saddle-shaped solution. In Section \ref{Sec:Layer} we establish some monotonicity properties of the layer solution $u_0$, as well as the pointwise estimate for the saddle solution in terms of the layer $u_0$, stated in Proposition~\ref{Prop:SaddleUnderLayer}. In Section~\ref{Sec:Monotonicity} we prove the monotonicity and second derivative properties of the saddle solution stated in Proposition~\ref{Prop:MonotonicityProperties}. Finally, Section~\ref{Sec:Stability} concerns the proof of the stability results, Theorem~\ref{Thm:Stability} and  Corollary~\ref{Cor:SimonsConeStableDim14}.

\section{Maximum principles for the linearized operator}
\label{Sec:MaximumPrinciple}

In this section we establish Proposition~\ref{Prop:MaxPrincipleLinearizedOperator}, a maximum principle for the linearized operator. To prove it, we follow the ideas appearing in \cite{Cabre-Saddle}, where an analogous maximum principle is proved for the local case $\s = 1$. The proof for the Laplacian uses a maximum principle in ``narrow'' sets (see for instance \cite{Cabre-Topics,BerestyckiNirembergVaradhan}). In our case, the use of the extension problem requires a similar maximum principle but in pairs of sets that we will call ``extension-narrow''\!, defined next.

\begin{definition}[``Extension-narrow'' pair of sets]
	Let $\Omega \subset \R^{n+1}_+$ be an open set, not necessarily bounded, and let $\Gamma \subset \partial_0 \Omega$ be nonempty ---recall that $\partial_0 \Omega$ is defined by \eqref{Eq:DefBoundaries}. Given $\theta \in (0,1)$ and $a\in (-1,1)$, we define $R_a(\Omega,\Gamma,\theta) \in (0, +\infty]$ to be the smallest positive constant $R$ for which
	\begin{equation}
	\label{Eq:Narrow}
	\dfrac{|B^+_R(x)\setminus \Omega|_a}{|B^+_R(x)|_a} \geq \theta \quad \text{ for every } x \in \Gamma\,,
	\end{equation}
	where 
	$$
	|E|_a := \int_E \lambda^a \dx \d \lambda\,.
	$$
	We say that $R_a(\Omega,\Gamma,\theta) = + \infty$ if no such radius exists.
	
	From this definition, we will say that a pair $(\Omega, \Gamma)$ is ``extension-narrow'' if $R_a(\Omega,\Gamma,\theta)$ is small enough depending on certain quantities.
\end{definition}

Note that if in \eqref{Eq:Narrow} we consider $a=0$ and full balls centered at every point $x\in \Omega$, we recover the usual definition of ``narrow'' set. Here, instead, we only consider half-balls centered at points $x\in \Gamma \subset \partial_0\Omega$.

Let us remark that both sets $\Omega\in \R^{n+1}_+$ and $\Gamma\subset \partial_0\Omega$ play an important role in this notion of ``narrowness'', as illustrated in the following examples in $\R^2_+$. On the one hand, let $ \Omega_1 = \{\lambda>(x-\varepsilon)/2\} \cap \{\lambda>-(x+\varepsilon)/2\} \cap \{\lambda>0\}\subset \R^2_+ $ and let $\Gamma_1 = \partial_0 \Omega_1 = (-\varepsilon,\varepsilon)\subset \R$ ---see Figure~\ref{Fig:Examples}~(a). This pair has $R_0(\Omega_1,\Gamma_1,1/2) = +\infty$ for all $\varepsilon > 0$ even though $\Gamma_1$ is ``narrow'' in $\R$ in the usual sense if $\varepsilon$ is small enough. On the other hand, the pair consisting of $\Omega_2 = \{0<\lambda<\varepsilon\}$ and $\Gamma_2 = \partial_0 \Omega_2 = \R$ is ``extension-narrow'' if $\varepsilon$ is small enough, while $\Gamma_2$ is not ``narrow'' in the usual sense in $\R$ ---see Figure~\ref{Fig:Examples}~(b). 

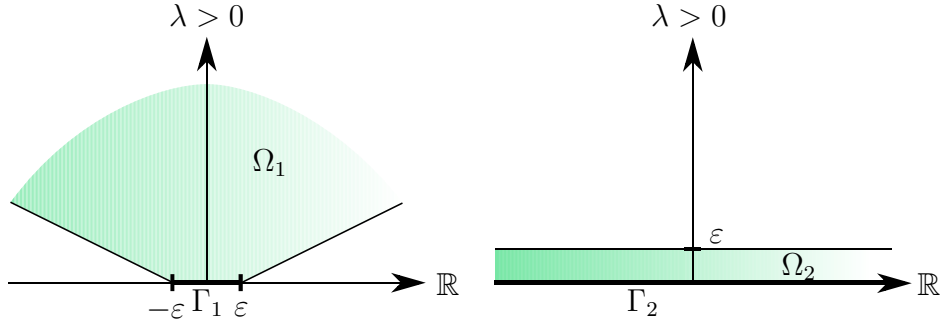
\begin{figure}
	\input{Examples.tex}
	\caption{(a) A pair $(\Omega_1, \Gamma_1)$ satisfying $R_0(\Omega_1,\Gamma_1,1/2) = +\infty$ but with $\Gamma_1$ being ``narrow'' in $\R$. (b) An ``extension-narrow'' narrow pair $(\Omega_2, \Gamma_2)$ with $\Gamma_2$ not  ``narrow'' in $\R$.}
	\label{Fig:Examples}
\end{figure}

Once the quantity  $R_a(\Omega,\Gamma, \theta)$ is defined, we can state precisely the maximum principle in ``extension-narrow'' pairs.

\begin{proposition}[Maximum principle in ``extension-narrow'' pairs]
	\label{Prop:MaxPrincipleNarrow}
	Let $\Omega \subset \R^{n+1}_+$ be an open set and let $\Gamma \subset \partial_0 \Omega$ be nonempty. Assume that there exists a nonempty open cone $E\subset \partial_0 \R^{n+1}_+ = \R^n$ such that $(E \times (0,+\infty))\cap \Omega=\varnothing$. 
	
	Let $a\in (-1,1)$ and let $v \in C^2 (\Omega)\cap C(\overline{\Omega})$ be a function bounded from above such that $\lambda^a v_\lambda \in C (\overline{\Omega})$, and assume that it satisfies
	\begin{equation}
	\label{Eq:MaxPrincipleNarrow}
	\beqc{\PDEsystem}
	-\div(\lambda^a \nabla v) &\leq& b(x,\lambda) v & \text{ in } \Omega\,, \\
	\dfrac{\partial v}{\partial \nu^a}  + c(x) v &\leq & 0 & \text{ on } \Gamma \,, \\
	v & \leq & 0 & \text{ on } \partial \Omega \setminus \Gamma \,,
	\eeqc
	\end{equation}
	where $b \leq 0$ in $\Omega$ and $c$ is bounded from below on $\Gamma$.
	
	Then, for every $\theta \in (0,1)$ there exists a constant $R^*$, depending only on $n$, $a$, $\theta$, and $\norm{c_-}_{L^\infty(\Gamma)}$, such that $v\leq 0 $ in $\Omega$ whenever $R_a(\Omega,\Gamma,\theta) \leq R^*$.
\end{proposition}

Before proving this result, let us explain why we need to introduce the notion of ``extension-narrowness''\!. In the proof of Proposition~\ref{Prop:MaxPrincipleLinearizedOperator} we will use this maximum principle in a pair $(\Omega, \Gamma)$ with $\Omega \subset \ocal \times (0,+\infty)$, and $\Gamma \subset \partial_0 \Omega$ in an $\varepsilon$-neighborhood in $\ocal$ of the cone $\ccal$ . In this case, $\Omega$ could  be very big (and not ``narrow'' in the usual sense) in $\R^{2m+1}_+$, as in Figure~\ref{Fig:ExtensionNarrow}. However, $\ocal^c\times (0,+\infty)$ is contained in the complement of $\Omega$ --- even if $\Omega$ filled all $\ocal\times (0,+\infty)$. Thus, it follows readily that $(\Omega,\Gamma)$ is ``extension-narrow'' by using that balls in this notion are centered in $\Gamma$ (see Corollary~\ref{Cor:MaxPrincipleNarrowSaddle} below for the details).

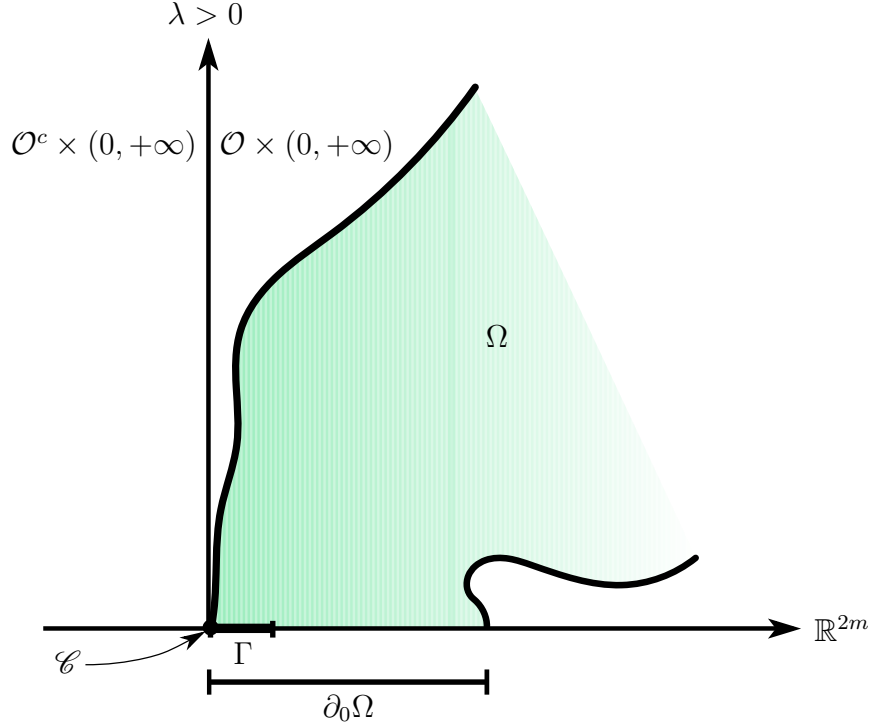
\begin{figure}
	\input{ExtensionNarrow.tex}
	\caption{An example of a pair $(\Omega,\Gamma)$ which is ``extension-narrow''\!.}
	\label{Fig:ExtensionNarrow}
\end{figure}

To prove Proposition~\ref{Prop:MaxPrincipleNarrow} we need the following weak Harnack inequality.

\begin{proposition}[Proposition 3.2 of \cite{TanXiong}]
	\label{Prop:WeakHarnack}
	Let $v \in H^1(B_R^+, \lambda^a)$ be a nonnegative function that weakly satisfies 
	$$
	\beqc{\PDEsystem}
	-\div(\lambda^a \nabla v) &\geq& 0 & \textrm{ in } B_R^+\,, \\
	\dfrac{\partial v}{\partial \nu^a} &\geq & 0 & \textrm{ on }  \partial_0 B_R^+\,.
	\eeqc
	$$
	
	Then, there exists a constant $p_0 > 0$, depending only on $n$ and $a$, such that for all $p\leq p_0$,
	\begin{equation}
	\label{Eq:WeakHarnack}
	\left( \int_{B_{R/2}^+} \lambda^a v^p \dx \d \lambda \right )^{1/p} \leq C_h R^{\frac{n+1+a}{p} } \inf_{B^+_{R/4}} v\,,
	\end{equation}
	for a positive constant $C_h$ depending only on $n$ and $a$.
\end{proposition}

With this result available, we can now present the proof of the maximum principle in ``extension-narrow'' pairs.

\begin{proof}[Proof of Proposition~\ref{Prop:MaxPrincipleNarrow}] 
	Define the sets 
	$$
	\Omega_+ := \left\{(x,\lambda) \in \Omega \ : \ v(x,\lambda)>0 \right\} 
	\quad \text{ and } \quad 
	\Gamma_+ := \partial \Omega_+ \cap \Gamma\,,
	$$ 
	and by contradiction assume that $\Omega_+$ is nonempty. Then, since $b\leq 0$,  $v$ satisfies
	$$
	\beqc{\PDEsystem}
	-\div(\lambda^a \nabla v) &\leq& 0 & \text{ in } \Omega_+\,, \\
	\dfrac{\partial v}{\partial \nu^a}  + c(x) v &\leq & 0 & \text{ on } \Gamma_+ \text{ (if this set is nonempty)}\,, \\
	v & \leq & 0 & \text{ on } \partial \Omega_+ \setminus \Gamma_+\,.
	\eeqc
	$$
	Now, we proceed in two steps in order to arrive at a contradiction.
	
	\textbf{Step 1.}
	First, we claim that if $\Gamma_+$ is nonempty then $\sup_{\Omega_+} v = \sup_{\Gamma_+} v $. That is, if we call 
	$$
	\overline{v} := v - \sup_{\Gamma_+} v,
	$$
	we then have $\overline{v} \leq 0$ in $\Omega_+$. To prove this, we use a classical Phragmen-Lindel\"of-type argument, as follows. Similar methods appear, among many others, in the proof of Theorem~1.2 of \cite{BerestyckiCaffarelliNiremberg-Monotonicity}, or Section~2.4 of \cite{CabreSolaMorales}.
	
	We now claim that, since the cone $E$ is open, there exists a nonempty open cone $F \subset E$ satisfying
	\begin{equation}
	\label{Eq:DistanceCones}
	|x-y| \geq c_0 > 0 \quad \text{ for every } x\in E^c \ \text{ and } y \in \overline{F}\,,
	\end{equation}
	for some positive constant $c_0$.
	
	Indeed, since $E$ is an open cone (with vertex, say, $z\in \partial E$), there exists a circular cone $E'\subset E$ with the same vertex $z$. Then, by sliding this circular cone in the direction of its axis, which can be assumed to be $e_n = (0,...,0,1)$, we obtain a new open cone $F\subset E$. Let us now show \eqref{Eq:DistanceCones}. Since $F \subset E' \subset E$, it is enough to prove \eqref{Eq:DistanceCones} for $x\in \partial E'$ and $ y\in\partial F$. Hence, we have
	$$x_n-z_n = \omega |x'-z'| \quad \text{ and } \quad y_n-z_n = \tau + \omega |y'-z'|, $$
	for some positive constants $\omega$ and $\tau$. Here, we are using the notation $z = (z',z_n)$. Now, if we call $\sigma = |x'-z'|-|y'-z'|$, we have $|x'-y'|\geq |\sigma|$ and thus
	\begin{align*}
	|x-y|^2 &= |x'-y'|^2 + |x_n-y_n|^2 \geq \sigma^2 + (\omega \sigma - \tau)^2   \\
	&= \left ( \sqrt{1+\omega^2} \sigma - \dfrac{\omega\tau}{\sqrt{1+\omega^2}}\right )^2 + \dfrac{\tau^2}{1+\omega^2}\geq\dfrac{\tau^2}{1+\omega^2},
	\end{align*}
	where the last constant is in fact the minimum distance between points on $\partial E'$ and $\partial F$. 
	
	Now, without loss of generality, we may assume that the vertex of $F$ is the origin. Let $F'$ be an open cone with the same vertex as $F$, and such that $\overline{F'\cap \Sph^{n-1}} \subset F\cap \Sph^{n-1}$. Let $\phi$ be the first eigenfunction of the Laplace-Beltrami operator in $\Sph^{n-1} \setminus \overline{F'}\subset \R^n$ with zero Dirichlet boundary conditions on $\partial F' \cap \Sph^{n-1}$, and let $\mu>0$ be its associated eigenvalue. Since $\partial F' \cap \Sph^{n-1}$ is contained in $F$, there exists a positive constant $\delta$ such that $\phi\geq \delta > 0$ in $\Sph^{n-1} \setminus \overline{F}$. Now, define the auxiliary function
	$$ 
	\psi(x,\lambda) = (1+\lambda^{2\s}) |x|^\beta \phi(x/|x|), 
	$$
	where $\beta$ is a positive real number and $\s = (1-a)/2 \in (0,1)$. Then, $\phi(x/|x|)\geq \delta$ for each $(x,\lambda) \in \Omega_+$,  since $x/|x| \in \Sph^{n-1} \setminus \overline{F}$. Moreover, by \eqref{Eq:DistanceCones} with $y=0$, we deduce that
	$$ 
	\psi(x,\lambda) \geq \delta (1+\lambda^{2\s}) |x|^\beta \geq \delta c_0^\beta > 0 \ \ \text{ in } \Omega_+,
	$$
	since $0$ is the vertex of $F$.
	On the other hand, note that if we choose $\beta>0$ solving $\beta(\beta+n-2)=\mu$, we have that $\psi$ satisfies
	$$
	\beqc{\PDEsystem}
	-\div(\lambda^a \nabla \psi) &=& 0 & \text{ in } \Omega_+\,, \\
	\ds \lim_{(x,\lambda)\in \Omega_+, \ |(x,\lambda)|\to +\infty} \psi & = & +\infty.
	\eeqc
	$$
	Thus, if we define 
	$$\overline{w}:=\dfrac{\overline{v} }{\psi} = \dfrac{v - \sup_{\Gamma_+} v}{\psi}\,,$$
	proving that $\overline{v} \leq 0$ in $\Omega_+$ is equivalent to showing that $\overline{w}\leq 0$ in $\Omega_+$, since $\psi$ is positive. Now, since $\sup_{\Gamma_+} v \geq 0$, it is easy to show that $\overline{w}$ satisfies
	\begin{equation}
	\label{Eq:wTildeEquationNarrow}
	\beqc{\PDEsystem}
	-\div(\lambda^a \nabla \overline{w})-2\lambda^a\dfrac{\nabla \psi}{\psi}\cdot\nabla \overline{w} &\leq& 0 & \text{ in } \Omega_+\,, \\
	\overline{w} &\leq& 0 & \text{ on } \partial \Omega_+\,, \\
	\ds \lim_{(x,\lambda)\in \Omega_+, \ |(x,\lambda)|\to +\infty} \overline{w} & \leq & 0\,.
	\eeqc
	\end{equation}
	Then, by the classical maximum principle we deduce that $\overline{w}\leq 0$ in $\Omega_+$, which yields $\overline{v}\leq 0$ in $\Omega_+$.
	
	Note that if $\Gamma_+$ is empty, the same argument applied to $v$ instead of $\overline{v}$ yields a contradiction with the assumption that $\Omega_+$ is nonempty. From now on in this proof, we will assume that $\Gamma_+ \neq \varnothing$.
	
	\textbf{Step 2.} By Step 1 and the definition of $\Omega_+$, we have that
	\begin{equation}
	\label{Eq:PositiveSup}
	M := \sup_{\Gamma_+} v > 0\,.
	\end{equation}
	Therefore, since $v\leq 0$ on $\partial \Omega_+\setminus \Gamma_+$, there exists a sequence $(x_k,0)\in \Gamma_+$ such that
	$$
	v(x_k) = v(x_k,0) \geq M \bpar{1- \dfrac{1}{k}}\,,
	$$
	where we are identifying $v$ with its trace on $\R^n$ to simplify the notation. 
	
	Now, given any $R>0$, let $\overline{c}_{n,\s}$ be the constant such that
	$$
	\fraclaplacian \{ \overline{c}_{n,\s} (R^2 - |x-x_k|^2)^\s_+ \} = 1 \quad \text{ in } B_R (x_k)\,,
	$$
	(see \cite{BogdanEtAl} for its explicit value) and take $\phi = \phi(x,\lambda)$ to be the $\s$-harmonic extension of 
	$$
	\phi(x,0) = c_1 M  \overline{c}_{n,\s}  (R^2 - |x-x_k|^2)^\s_+\,,
	$$
	where $c_1$ is a positive constant to be chosen later. Thus, $\phi$ solves
	$$
	\beqc{\PDEsystem}
	-\div(\lambda^a \nabla \phi) &=& 0 & \text{ in } B^+_R(x_k)\,, \\
	\dfrac{\partial \phi}{\partial \nu^a} & = & \dfrac{c_1 M}{d_\s} & \text{ on } \partial_0 B^+_R(x_k)\,.
	\eeqc
	$$
	Moreover, on $\partial_0 B^+_R(x_k) \cap \Gamma_+$ we have
	$$
	\dfrac{\partial v}{\partial \nu^a} \leq - c v  \leq \norm{c_-}_{L^\infty (\Gamma)} v \leq  \norm{c_-}_{L^\infty (\Gamma)} M \leq \dfrac{\partial \phi}{\partial \nu^a}
	$$
	if we choose $c_1 > d_\s\norm{c_-}_{L^\infty (\Gamma)}$\,.
	
	Thus, $v-\phi$ is $\s$-subharmonic in $B^+_R(x_k) \cap  \Omega_+$ and has a nonpositive flux on $\partial_0 B^+_R(x_k) \cap \Gamma_+$. In addition, $v-\phi \leq v \leq 0$ in $B_R^+(x_k)\cap (\partial  \Omega_+ \setminus \Gamma_+)$. Therefore, its positive part $(v-\phi)_+$ extended to be zero in $B_R^+(x_k)\setminus  \Omega_+ $ is a continuous function which is $\s$-subharmonic in $B^+_R(x_k)$ and has a nonpositive flux on $\partial_0 B^+_R(x_k)$, both properties in a weak sense.
	
	We define $w := M -(v-\phi)_+$, which is a continuous nonnegative function and satisfies in a weak sense
	$$
	\beqc{\PDEsystem}
	-\div(\lambda^a \nabla w) &\geq & 0 & \text{ in } B^+_R(x_k) \,, \\
	\dfrac{\partial w}{\partial \nu^a} & \geq & 0 & \text{ on } \partial_0 B^+_R(x_k)  \,.
	\eeqc
	$$
	Hence, $w$ fulfills the hypotheses of Proposition~\ref{Prop:WeakHarnack} and thus \eqref{Eq:WeakHarnack} holds. As a consequence, if we take $R = 2\,R_a( \Omega,\Gamma, \theta)$ and $p$ as in \eqref{Eq:WeakHarnack}, we have
	\begin{align*}
	\theta^{1/p} M & \leq \left (  \dfrac{|B^+_{R/2}(x_k)\setminus   \Omega|_a}{|B^+_{R/2}(x_k)|_a}  M^p \right)^{1/p} \\
	& \leq \left (  \dfrac{|B^+_{R/2}(x_k)\setminus  \Omega_+|_a}{|B^+_{R/2}(x_k)|_a}  M^p \right)^{1/p} \\
	&= \dfrac{1}{|B^+_{R/2}(x_k)|_a^{1/p}}  \left (  \int_{B^+_{R/2}(x_k)\setminus  \Omega_+} \lambda ^a M^p \d x \d \lambda  \right)^{1/p} \\
	&\leq |B_1^+|_a^{-1/p} (R/2)^{- \frac{n+1+a}{p}} \left (  \int_{B^+_{R/2}(x_k)} \lambda ^a w^p \d x \d \lambda  \right)^{1/p} \\
	&\leq 2^{\frac{n+1+a}{p}}|B_1^+|_a^{-1/p} C_h \inf_{B^+_{R/4}(x_k)} w \\
	& \leq 2^{\frac{n+1+a}{p}}|B_1^+|_a^{-1/p} C_h \,  w(x_k).
	\end{align*}
	Here we have used the definition of $R_a( \Omega,\Gamma,\theta)$, the fact that $w \equiv M$ in $B^+_R(x_k)\setminus \Omega_+$, the scaling properties of $|\cdot |_a$ and the weak Harnack inequality \eqref{Eq:WeakHarnack}. 
	
	Now, if $c_1 \overline{c}_{n,\s} R^{2\s} \leq 1/2$, then $w(x_k) = M - v(x_k) + \phi(x_k)$ for $k$ large enough. Therefore, for such indices $k$ we conclude
	\begin{align*}
	\theta^{1/p} M &\leq  2^{\frac{n+1+a}{p}}|B_1^+|_a^{-1/p} C_h \{M - v(x_k) + \phi(x_k) \} \\
	& \leq 2^{\frac{n+1+a}{p}}|B_1^+|_a^{-1/p} C_h \{1/k + c_1  \overline{c}_{n,\s} R^{2\s} \}M\,.
	\end{align*}
	Hence, if we take $R_a( \Omega,\Gamma, \theta)$ small enough such that $c_1 \overline{c}_{n,\s} (2 R_a( \Omega,\Gamma, \theta))^{2\s} < 1$ and $ 2^{\frac{n+1+a}{p}} |B_1^+|_a^{-1/p} C_h c_1 \overline{c}_{n,\s} (2 R_a( \Omega,\Gamma, \theta))^{2\s} < \theta^{1/p}$, we get that
	$$
	M \bpar{1- \dfrac{C}{k}}\leq 0
	$$
	for some positive constant $C$ independent of $k$. Letting $k\to +\infty$, this leads to $M \leq 0$, which contradicts \eqref{Eq:PositiveSup}. 
	
	Therefore, our initial assumption stating $ \Omega_+ \neq \varnothing$ is false. This means that $v\leq 0$ in $\Omega$.
\end{proof}

\begin{remark}
	\label{Remark:WeakMPNarrow}
	It will be useful later to note that Proposition~\ref{Prop:MaxPrincipleNarrow} (and as a consequence, Proposition~\ref{Prop:MaxPrincipleLinearizedOperator}) is also valid not requiring $v$ to be $C^2$ in the whole $\Omega$. Indeed, we only need to assume that $v\in C(\Omega)$, that the equation $\div (\lambda^a \nabla v) \leq b(x,\lambda)v$ holds pointwise where $v$ is regular, and that $v$ cannot have a local maximum at a nonregular point.
	
	This will be used in the proof of Proposition~\ref{Prop:SaddleUnderLayer} with $v = u- CU$ in $\Omega = \ocal \times (0,+\infty)$, where $u$ is a saddle-shaped solution, $U$ is defined by \eqref{Eq:DefULayer}, and $C$ is a positive constant. Note that $U$ is Lipschitz but not $C^2$ across $\{t=0, \lambda\geq 0\}$. Therefore, as we will see in Section~\ref{Sec:Layer}, $U$ is only $\s$-superharmonic (pointwise) in $\Omega \setminus \{t=0, \lambda\geq 0\}$. Nevertheless, by this remark, Proposition~\ref{Prop:MaxPrincipleNarrow} will hold in this case thanks to the fact that the graph of $v = u-CU$ in its nonregular points makes the ``good angle'' for the maximum principle to hold (see the proof of Proposition~\ref{Prop:SaddleUnderLayer} for the details).
\end{remark}

As a consequence of Proposition~\ref{Prop:MaxPrincipleNarrow}, next we establish that the maximum principle holds in pairs $(\Omega, \Gamma)$ with $ \Omega \subset \ocal \times (0, +\infty) \subset \R^{2m+1}_+$ and $\Gamma \subset \partial_0  \Omega$ lying in an $\varepsilon$-neighborhood of the Simons cone.

\begin{corollary}
	\label{Cor:MaxPrincipleNarrowSaddle}
	Let $\Omega\subset \ocal \times (0, +\infty) \subset \R^{2m + 1}_+$ and let $\Gamma \subset \partial_0 \Omega$ be nonempty. Assume that $\Gamma \subset \ncal_\varepsilon := \{t < s < t+\varepsilon, \ \lambda= 0 \}$. 
	
	Then, if $\varepsilon$ is small enough, depending only on $n$, $\s$, and $\norm{c_-}_{L^\infty(\Gamma)}$, the maximum principle holds in $\Omega$ in the sense of Proposition~\ref{Prop:MaxPrincipleNarrow}. That is, if $v \in C^2 (\Omega)\cap C(\overline{\Omega})$ is bounded from above, $\lambda^a v_\lambda \in C (\overline{\Omega})$, and $v$ satisfies \eqref{Eq:MaxPrincipleNarrow}, then $v\leq 0$ in $\Omega$.
\end{corollary}

To prove it, it is enough to realize that the Simons cone separates every ball centered at a point in the cone into two regions with comparable measure. In fact, it is interesting to note that these two regions have exactly the same measure, as stated next.

\begin{lemma}
	\label{Lemma:HalfBallSimonsCone}
	Let $x_0\in \ccal\subset \R^{2m}$. Then, 
	$$ | B_r(x_0)\cap \ocal| = | B_r(x_0)\setminus \ocal| = \frac{1}{2} |B_r(x_0)| \ \ \textrm{ for all} \ \ r>0\,. $$
\end{lemma}

This result was stated in \cite{Cabre-Saddle}, but without a proof. For the sake of completeness, we include here a simple one.

\begin{proof}[Proof of Lemma~\ref{Lemma:HalfBallSimonsCone}]
	First, let us call $\ical := \R^{2m}\setminus \overline{\ocal}$. Since $x_0\in \ccal$, we have that $x_0 = (x_0',x_0'')\in \R^m \times \R^m$ satisfies that $|x_0'|=|x_0''|$. Therefore, there exists an orthogonal transformation $R\in O(m)$ such that $R  x_0' = x_0''$. Let us define $\overline{R}:\R^{2m}\to\R^{2m}$ by $\overline{R}(x',x'') = (R x',x'')$, which is a linear isometry that keeps invariant $\ocal$ and $\ical$. With these properties it is easy to check that for every $y\in \R^{2m}$ it holds
	\begin{equation}
	\label{Eq:IsometrySimonsCone1}
	|B_r(y)\cap \ical| = |\overline{R} \left( B_r(y)\cap \ical\right)| = | B_r(\overline{R}y) \cap \ical |\,,
	\end{equation}
	and the same replacing $\ical$ with $\ocal$.
	
	On the other hand, let us define $S:\R^{2m}\to\R^{2m}$ by $S(x',x'') = (x'',x')$, which is also a linear isometry and transforms $\ocal$ into $\ical$ and vice versa. Therefore, for every $y\in \R^{2m}$ we have
	\begin{equation}
	\label{Eq:IsometrySimonsCone2}
	|B_r(y)\cap \ical| = |S \left( B_r(y)\cap \ical\right)| = | B_r(Sy) \cap \ocal |\,.
	\end{equation}
	
	Finally, note that by the definition of $\overline{R}$, it is satisfied $S\overline{R} x_0 = \overline{R} x_0$. By combining this with \eqref{Eq:IsometrySimonsCone1} and \eqref{Eq:IsometrySimonsCone2} applied to $y=x_0$ and $y=\overline{R}x_0$ respectively, we obtain
	$$ |B_r(x_0)\cap \ical| = | B_r(\overline{R}x_0) \cap \ical | = | B_r(S\overline{R}x_0) \cap \ocal | = | B_r(\overline{R}x_0) \cap \ocal | = |B_r(x_0)\cap \ocal|\,. $$
\end{proof}

With this lemma available we proceed with the proof of Corollary~\ref{Cor:MaxPrincipleNarrowSaddle}.

\begin{proof}[Proof of Corollary~\ref{Cor:MaxPrincipleNarrowSaddle}]
	Note that $\R^{2m} \setminus \overline{\ocal}$ is an open cone outside $\ocal$, and thus $\{ (\R^{2m} \setminus \overline{\ocal}) \times (0, +\infty) \} \cap \Omega $ is empty. Hence, we can use Proposition~\ref{Prop:MaxPrincipleNarrow} by noticing that,  if we take $\theta = 2^{-4m - 3-2a}$, then $R_a(\Omega,\Gamma,\theta)\leq \varepsilon$. Indeed, recall first that by Lemma~4.2 in \cite{CabreTerraI}, $|s-t|/\sqrt{2}$ is the distance to the cone. Then, let $x\in \Gamma$ and let $\overline{x}\in \ccal$ a point realizing this distance. Since $x\in \Gamma \subset \ncal_\varepsilon$, we have that $|x-\overline{x}| \leq \varepsilon /\sqrt{2} < 3\varepsilon /4$ and therefore
	$$
	B_{\varepsilon/4}^+ (\overline{x})\setminus \big( \ocal \times (0, +\infty)\big) \subset B_{\varepsilon/4}^+ (\overline{x})\setminus \Omega \subset B_\varepsilon^+ (x)\setminus \Omega\,.
	$$
	Hence, by the scaling properties of $|\cdot|_a$ and Lemma~\ref{Lemma:HalfBallSimonsCone} ---used at each level $\{\lambda = \lambda_0\}$, with $\lambda_0\in (0,\varepsilon/4)$---, we have
	$$
	2^{-4m - 3-2a} |B^+_\varepsilon (x)|_a = \dfrac{1}{2} |B_{\varepsilon/4}^+(\overline{x}) |_a = |B_{\varepsilon/4}^+ (\overline{x})\setminus \big( \ocal \times (0, +\infty)\big)  |_a \leq |B_\varepsilon^+ (x)\setminus \Omega |_a\,.
	$$
\end{proof}

With this result at hand we can now establish the maximum principle for the linearized operator in $\ocal\times (0,+\infty)$ at a saddle-shaped solution.

\begin{proof}[Proof of Proposition~\ref{Prop:MaxPrincipleLinearizedOperator}]
	Let $u$ be a saddle-shaped solution. A key point in the proof is that $u$ is a positive supersolution in $\ocal\times(0,+\infty)$ of the linearized problem at $u$. Indeed, since $u>0$ in $\partial_0\Omega \subset \ocal$,
	\begin{equation}
	\label{Eq:uSupersolLinearized}
	\mathscr{L}_u u = d_\s \dfrac{\partial u}{\partial \nu^a}  -f'(u) u = f(u) - f'(u) u > 0\quad \text{ on } \partial_0\Omega \,.
	\end{equation}
	We have used that since $f''<0$ in $(0,1)$ and $f(0)=0$, it satisfies $f'(\tau)\tau < f(\tau)$ for all $\tau\in (0,1)$.

	Now, we define
	$$
	w := \dfrac{v}{u}\,.
	$$
	Note that $w$ is well defined in $\Omega$, since $u$ is positive in such set. The usual strategy (see \cite{BerestyckiNirembergVaradhan}) in some proofs of the maximum principle is to assume that the supremum of $w$ in $\Omega$ is positive and then arrive at a contradiction.  Nevertheless, a priori we do not know that $\sup_\Omega w< +\infty$, since $u$ vanishes on $\ccal\times [0,+\infty)$ and $\partial \Omega$ could intersect this set. Thus, in the following arguments we will consider the supremum of $w$ in a subset of $\partial_0 \Omega$ that is at a positive distance to the zero level set of $u$. Then, using the maximum principle in ``extension-narrow'' pairs we will see that, assuming this supremum to be positive, it will indeed agree with the supremum in the whole set $\Omega$ (see the details below). After some arguments, we will arrive at a contradiction. A similar strategy was used by Cabr\'e in \cite{Cabre-Saddle}, to prove an analogous maximum principle in the local case $\s = 1$.

	Les us proceed with the details. For $\varepsilon > 0$, set
	$$
	\ocal_\varepsilon := \{t +\varepsilon < s,\ \lambda = 0 \} \quad \textrm{ and } \quad \mathcal{N}_\varepsilon := \{t < s < t+\varepsilon ,\ \lambda = 0 \}\,,
	$$
	and take $\varepsilon$ small enough such that for each set $\Gamma \subset \partial_0 \Omega$ satisfying $\Gamma \subset \mathcal{N}_\varepsilon$, the pair $(\Omega, \Gamma)$  is ``extension-narrow''. Hence, the maximum principle, as in Corollary~\ref{Cor:MaxPrincipleNarrowSaddle}, holds for the pair $(\Omega, \Gamma)$. 
	
	Next, we claim that
	\begin{equation}
	\label{Eq:u>delta}
	u \geq \delta >0\ \ \text{ in }   \ocal_\varepsilon
	\end{equation}
	for some positive constant $\delta$. Indeed, thanks to the asymptotic behavior of $u$ (see part (ii) of Theorem~\ref{Th:Summary}), and since $U(x) \geq u_0(\varepsilon/\sqrt{2})$ for $x\in \ocal_\varepsilon$,  there exists a radius $R>0$ such that $u(x) \geq u_0(\varepsilon/\sqrt{2})/2$ if $|x|>R$ and $x\in \ocal_\varepsilon$. Since $u$ is positive in the compact set $\overline{\ocal_\varepsilon} \cap \overline{B_R}$, we conclude the claim.
	
	We define
	$$
	\Gamma := \partial_0 \Omega \cap \ncal_\varepsilon\,,
	$$
	and let
	$$
	S := \sup_{\partial_0 \Omega \cap \ocal_\varepsilon } w\,,
	$$
	which is finite by the fact that $u$ is bounded from below  by $\delta>0$ in $\ocal_\varepsilon$ and $v$ is bounded from above. Assume by contradiction that $S > 0$.
	
	First, we claim that $S = \sup_\Omega w$. To see this, we only need to show that $w \leq S$ in $\Omega$. Define $\varphi := v -Su$ and note that since $S\geq 0$, $\varphi$ satisfies 
	$$
	\beqc{\PDEsystem}
	-\div(\lambda^a \nabla \varphi) &\leq& b(x,\lambda) \varphi & \text{ in } \Omega \,, \\
	\dfrac{\partial \varphi}{\partial \nu^a}  &\leq & c(x) \varphi & \text{ on } \Gamma \,, \\
	\varphi & \leq & 0 & \text{ on } \partial \Omega \setminus \Gamma  \,,
	\eeqc
	$$
	with $c(x) = f'(u) / d_\s$. By the maximum principle in the ``extension-narrow'' pair $(\Omega,\Gamma)$, we have $\varphi \leq 0$ in $\Omega$, which yields $w = v/u \leq S$ in $\Omega$. Thus, the claim is proved.
	
	Now, by the hypothesis on $\partial_L \Omega$ and at infinity on $v$, and the fact that $u>\delta$ in $\ocal_{\varepsilon}$, we have that $S$ is attained at some point  $(x_0,0)\in\partial_0\Omega \subset \ocal$. At this point we have
	\begin{equation}
	\label{Eq:NormalDerivativeAtMaximum}
	\dfrac{\partial w}{\partial \nu^a}(x_0) = -\lim_{\lambda \downarrow 0} \lambda^a w_\lambda (x_0,\lambda) = \lim_{\lambda \downarrow 0} \dfrac{w(x_0,0) - w(x_0, \lambda)}{\lambda^{2\s}} \geq 0\,,
	\end{equation}
	since $w(x_0,0)$ is the maximum.
	
	On the other hand, observe that
	$$
	d_\s u^2 \dfrac{\partial w}{\partial \nu^a} = d_\s  \dfrac{\partial v}{\partial \nu^a} u  - d_\s  \dfrac{\partial u}{\partial \nu^a} v = u \mathscr{L}_u v  -  v \mathscr{L}_u u \leq -v \mathscr{L}_u u \quad \text{ on } \partial_0 \Omega \subset \ocal\,,
	$$
	since $u>0$ in $\ocal$ and $\mathscr{L}_u v \leq 0$ in $\partial_0 \Omega $. Therefore, at the point $x_0$ we have, using also \eqref{Eq:uSupersolLinearized}, 
	$$
	\dfrac{\partial w}{\partial \nu^a}(x_0) \leq -\dfrac{S}{d_\s u(x_0)} \mathscr{L}_u u(x_0) < 0\,,
	$$
	which contradicts \eqref{Eq:NormalDerivativeAtMaximum}. Note that in this last argument is crucial the fact that $x_0 \in \partial_0 \Omega \subset \ocal$ and thus $u(x_0)>0$ and $\mathscr{L}_u u (x_0) > 0$. 
	
	Hence, the assumption $S>0$ is false and therefore $w \leq 0$ in $\partial_0 \Omega \cap \ocal_\varepsilon $. Since $u > 0$ in $\ocal$, this yields that $v \leq 0$ in $\partial_0 \Omega \cap \ocal_\varepsilon $. Finally, by the maximum principle in the ``extension-narrow'' pair $(\Omega, \Gamma)$ applied to $v$, it follows that $v\leq 0$ in $\Omega$.
\end{proof}

\section{Uniqueness of the saddle-shaped solution}
\label{Sec:Uniqueness}
Thanks to the maximum principle in $\ocal\times (0,+\infty)$ for the linearized operator we can now establish the uniqueness of the saddle-shaped solution.

\begin{proof}[Proof of Theorem~\ref{Thm:Uniqueness}]
	Let $u_1$ and $u_2$ be two saddle-shaped solutions. Define $v := u_1 - u_2$, a function that depends only on $s$ and $t$ and that is odd with respect to $\ccal$. Then, $\div(\lambda^a \nabla v) = 0$ in $\ocal \times (0,+\infty)$, $v=0$ on $\partial_L \left( \ocal \times (0,+\infty) \right) = \ccal \times [0,+\infty)$ and
	$$
	d_\s \dfrac{\partial v}{\partial \nu^a} = f(u_1) - f(u_2) \leq f'(u_2) (u_1 - u_2) = f'(u_2) v \quad \textrm{ on } \ocal \times \{0\}\,,
	$$
	since $f$ is concave in $(0,1)$. Moreover, by the asymptotic result (see Theorem~\ref{Th:Summary}), we have
	$$
	\limsup_{x \in \ocal,\ |x|\to +\infty} v(x, 0) = 0\,.
	$$
	
	Finally, by the maximum principle for the linearized operator in $\ocal\times (0,+\infty)$, see Proposition~\ref{Prop:MaxPrincipleLinearizedOperator}, we deduce that $v \leq 0$ in $\ocal \times [0, +\infty)$, which yields $u_1 \leq u_2$ in $\ocal \times [0, +\infty)$. Interchanging $u_1$ and $u_2$, we obtain $u_1 \geq u_2$ in $\ocal \times [0, +\infty)$. Therefore, $u_1 = u_2$ in $\R^{2m+1}_+$.
\end{proof}

\section{The layer solution and a pointwise estimate for the saddle-shaped solution}
\label{Sec:Layer}

This section is devoted to establish some monotonicity properties of the layer solution $u_0$ and a pointwise estimate for the saddle-shaped solution (Proposition~\ref{Prop:SaddleUnderLayer}). We start with a maximum principle similar to Proposition~\ref{Prop:MaxPrincipleLinearizedOperator}, but for the linearized operator at $u_0$ in the set $\{u_0 > 0\}$, which plays the role that $\ocal\times(0,+\infty)$ had for the saddle-shaped solution.

\begin{proposition}
	\label{Prop:MaxPrincipleLinearizedOperator2D}
	Let $u_0:\overline{\R^2_+}\to \R$ be the layer solution of \eqref{Eq:LayerSolution} and let $\mathscr{L}_{u_0} $ be defined by 
	$$
	\mathscr{L}_{u_0} v := d_\s \dfrac{\partial v}{\partial \nu^a}  -f'(u_0) v\,\ \ \text{ on } \ \R=\partial_0\R^2_+\,.
	$$
	Let $\Omega \subset (0,+\infty) \times (0,+\infty)$ be an open set such that $\partial_0 \Omega$ is nonempty. 
	
	Let $v \in C^2 (\Omega)\cap C(\overline{\Omega})$ be bounded from above and satisfying $\lambda^a v_\lambda \in C (\overline{\Omega})$. Assume that
	$$
	\beqc{\PDEsystem}
	-\div(\lambda^a \nabla v) &\leq& b(x,\lambda) v & \text{ in } \Omega \subset (0, +\infty)\times (0,+\infty)\,, \\
	\mathscr{L}_{u_0} v &\leq & 0 & \text{ on } \partial_0 \Omega \subset (0,+\infty) \,, \\
	v & \leq & 0 & \text{ on } \partial_L \Omega\,,\\
	\ds \limsup_{x\in \partial_0 \Omega,\ |x|\to +\infty} v(x,0) & \leq & 0\,,
	\eeqc
	$$
	with $b \leq 0$. Then, $v\leq 0$ in $\Omega$.
\end{proposition}

\begin{proof}
	Since it is analogous (and simpler) to the proof of Proposition~\ref{Prop:MaxPrincipleLinearizedOperator}, we just sketch it here pointing out what needs to be adapted. The key fact is that $u_0$ is a positive supersolution to the linearized problem. This is an analogous situation to that of Proposition~\ref{Prop:MaxPrincipleLinearizedOperator} . That is, $u_0$ is $\s$-harmonic in $(0,+\infty)\times (0,+\infty)$, positive in $(0,+\infty)\times [0,+\infty)$, and
	\begin{equation}
	\label{Eq:u0SupersolLinearized}
	d_\s \dfrac{\partial u_0}{\partial \nu^a}  = f(u_0) > f'(u_0) u_0 \quad \textrm{ on } (0, +\infty)\times \{0\}\,,
	\end{equation}
	where we have used that $f''<0$ in $(0,1)$ and $f(0)=0$.

	Then, one defines $w := v/u_0$ and proceeds exactly as in the proof of Proposition~\ref{Prop:MaxPrincipleLinearizedOperator}, replacing  $u$ by $u_0$ in the whole argument, and also replacing  $\ocal_\varepsilon$ and $\ncal_\varepsilon$ by $(\varepsilon,+\infty)$ and $(0,\varepsilon)$ respectively. In addition, \eqref{Eq:u>delta} follows immediately from the fact that $u_0(x,0)$ is increasing. The rest of the proof is completely analogous by using \eqref{Eq:u0SupersolLinearized}.
\end{proof}

With this maximum principle we can now prove the following monotonicity and concavity properties of the layer solution.

\begin{lemma}
	\label{Lemma:MonotonicityLayer}
	Let $u_0$ be the layer solution of \eqref{Eq:LayerSolution}. Then,
	$$ \frac{\partial}{\partial x} u_0(x,\lambda) > 0 \ \ \ \text{in} \ \ \ \R\times [0,+\infty) $$
	and
	$$ \frac{\partial^2}{\partial x^2} u_0(x,\lambda) < 0 \ \ \ \text{in} \ \ \ (0,+\infty)\times [0,+\infty)\,. $$
\end{lemma}

\begin{proof}
	First of all, let us remark that $u_0$ has the required regularity to apply the following arguments by the results of \cite{CabreSireI} (see Section~\ref{Sec:Monotonicity} for more details in the more involved setting of the saddle-shaped solution).

	The monotonicity of the first derivative was already stated in Remark~4.7 of \cite{CabreSireI}, but we include here the short proof for completeness. By differentiating \eqref{Eq:LayerSolution} with respect to $x$, we obtain that $\div(\lambda^a \nabla (\partial_x u_0)) = 0$ in $\R\times (0,+\infty)$. Moreover,  $\partial_x u_0(x,0) > 0$ for $x\in \R$; see \eqref{Eq:LayerSolution}. Then, the result follows directly from the Poisson formula.
	
	Next, we show the second statement. If we call 
	$$
	v(x,\lambda) := \partial_{xx} u_0(x,\lambda)\,,
	$$
	by differentiating \eqref{Eq:LayerSolution} twice with respect to $x$, we get
	\begin{equation*}
	\beqc{\PDEsystem}
	\div(\lambda^a \nabla v) &=& 0 & \text{ in } (0,+\infty)\times (0,+\infty) \,, \\
	d_a\,\dfrac{\partial v}{\partial \nu^a} - f'(u_0) v = f''(u_0) (\partial_x u_0) ^2 &\leq & 0 & \text{ on } (0,+\infty) \times \{0\} \,, \\
	v & = & 0 & \text{ on } \{0\} \times (0,+\infty)\,.
	\eeqc
	\end{equation*}
	Notice that $v = 0$ on $\{0\} \times (0,+\infty)$ since $v$ is an odd function with respect to the first variable (recall that $u_0$ is odd in $x$).
	
	Moreover, by repeating the argument of Lemma~4.8 in \cite{CabreSireI} for $\partial_{xx} u_0$, it is easy to see that $\partial_{xx} u_0(x,0) \to 0$ as $|x|\to +\infty$. Therefore, by Proposition~\ref{Prop:MaxPrincipleLinearizedOperator2D} we deduce that $v\leq 0$ in $[0,+\infty)\times [0,+\infty)$. Finally, we get that it is in fact negative in $(0,+\infty)\times [0,+\infty)$ by applying the strong maximum principle.
\end{proof}

Now we prove that the function 
$$
U(s,t,\lambda) := u_0 \left( \frac{s-t}{\sqrt{2}}, \lambda\right)
$$ 
is a barrier for the saddle-shaped solution. To do it, we will use a maximum principle in $\ocal \times (0,+\infty)$ for the linearized problem at $U$.

\begin{proof}[Proof of Proposition~\ref{Prop:SaddleUnderLayer}]
	The idea is to repeat the arguments in the proof of Proposition~\ref{Prop:MaxPrincipleLinearizedOperator}, but using $U$ instead of $u$ as the positive supersolution to the linearized problem involving the operator 
	$$
	\mathscr{L}_U w:= d_\s \dfrac{\partial w}{\partial \nu^a}  -f'(U)w. 
	$$ 
	In order to do it, we need to point out several facts.
	
	First, note that 
	$$
	U\in C^2\big( (\ocal\times (0,+\infty))\setminus \{t=0,\ \lambda > 0\} \big ) \cap \Lip \big( \overline{\R^{2m+1}_+}\big )\,,
	$$
	and $U$ cannot have a local minimum at $\{t=0,\ \lambda \geq 0\}$. Indeed, for every $\lambda \geq 0$,
	$$
	\lim_{\tau\to 0^-} \partial_{x_{m+1}}  U(x_1,...x_m,\tau,0,...,0,\lambda) = \frac{1}{\sqrt{2}}\partial_xu_0\left(\frac{s}{\sqrt{2}},\lambda\right)>0\,, 
	$$
	and
	$$
	\lim_{\tau\to 0^+} \partial_{x_{m+1}} U(x_1,...x_m,\tau,0,...,0,\lambda) = -\frac{1}{\sqrt{2}}\partial_xu_0\left(\frac{s}{\sqrt{2}},\lambda\right) <0\,.
	$$
	Note that the same property concerning a local minimum at $\{t=0,\ \lambda \geq 0\}$ holds if we add to $U$ a regular function.
	
	Next, we claim that $U$ is a positive supersolution in $\ocal$ to the linearized problem for $\mathscr{L}_U$. Indeed, by the concavity of $f$, we have that
	$$
	\mathscr{L}_U U  = f(U) - f'(U)U> 0 \quad \text{ in } \ocal\,.
	$$
	Moreover, a simple computation in the $(s, t, \lambda)$ variables shows that
	\begin{equation}
	\label{Eq:EqForULayer}
	\div (\lambda^a \nabla U) = \lambda^a \dfrac{m-1}{\sqrt{2}} \dfrac{t-s}{st} \partial_x u_0 \bpar{\dfrac{s-t}{\sqrt{2}}, \lambda}  \quad \text{ in } \R^{2m + 1}_+ \setminus \{st=0,\ \lambda >0\}\,.
	\end{equation}
	Therefore, $U$ is $\s$-superharmonic in $ (\ocal \times (0, +\infty) )\setminus \{t=0,\ \lambda > 0\}$ ---recall that $ \partial_x u_0 > 0$ by Lemma~\ref{Lemma:MonotonicityLayer}.

	Now, we define 
	$$ v:= u - U \quad \text{ and } \quad \Omega :=  \ocal \times (0, +\infty)\,,
	$$
	and we want to see that $v\leq 0$ in $\Omega$. First, since $u$ is $\s$-harmonic, we have that 
	$$ 
	-\div (\lambda^a \nabla v) \leq 0 \quad \text{ in }\Omega \setminus \{t=0,\ \lambda > 0\}
	$$
	and that $v$ cannot have a local maximum at $\{t=0,\ \lambda \geq 0\}$. In addition, both $u$ and $U$ vanish at $\ccal \times [0,+\infty)$ and by the asymptotic behavior of $u$ (see Theorem~\ref{Th:Summary}), we have $\lim_{x\in  \ocal,\ |x|\to +\infty} v(x,0) = 0\,$. On the other hand, since $f$ is concave in $(0,1)$, we get 
	$$
	d_\s \dfrac{\partial v}{\partial \nu^a}  = f(u) - f(U) \leq f'(U)v \quad \text{ on } \partial_0 \Omega\,.
	$$

	Collecting all these facts, we can repeat the proof of Proposition~\ref{Prop:MaxPrincipleLinearizedOperator}, using $U$ instead of $u$ as the positive supersolution to the linearized problem for $\mathscr{L}_U$ to see that $v\leq 0$ in $\Omega$. All the arguments are analogous, taking into account Remark~\ref{Remark:WeakMPNarrow} when using the maximum principles in ``extension-narrow'' pairs. Therefore, we conclude that $v\leq 0$ in $\Omega$ and, by the odd symmetry of $u$ and $U$, we get \eqref{Eq:SaddleUnderLayer}.
\end{proof}

\section{Monotonicity properties}
\label{Sec:Monotonicity}

In this section we establish the monotonicity properties of $u$ stated in Proposition~\ref{Prop:MonotonicityProperties}. For this, we will apply the maximum principle of Proposition~\ref{Prop:MaxPrincipleLinearizedOperator} to some derivatives of $u$. Therefore, we need some regularity results that we collect next. 

Recall that we assume that $f\in C^{2,\alpha}$ for some $\alpha\in(0,1)$. Since $u$ is a bounded solution to the first equation in \eqref{Eq:AllenCahnWithExtension}, then $u\in C^\infty(\R^{2m+1}_+)$. Regarding the regularity on $\{\lambda = 0\}$, $u(\cdot,0)\in C^{2,\alpha}(\R^{2m})$ by applying Lemma 4.4 from \cite{CabreSireI}. Moreover, \cite{CabreSireI} also gives the following uniform bound:
$$ 
||u||_{C^\alpha\left(\overline{\R^{2m+1}_+}\right)} + ||\nabla_x u||_{C^\alpha\left(\overline{\R^{2m+1}_+}\right)} + ||D^2_x u||_{C^\alpha\left(\overline{\R^{2m+1}_+}\right)} \leq C, 
$$
for some $C>0$ depending only on $m$, $\s$ , $||f||_{C^{2,\alpha}}$, and $||u||_{L^\infty(\R^{2m+1}_+)}$.

Next, since the horizontal first derivatives of $u$ satisfy $\div(\lambda^a \, \nabla u_{x_i}) = 0$ and also $d_\s \partial_{\nu^a}  u_{x_i} = f'(u) \,u_{x_i} \in C^\alpha(\R^{2m})$, and the  horizontal second derivatives of $u$ satisfy $\div(\lambda^a \, \nabla u_{x_i x_j}) = 0$ and also  $d_\s\partial_{\nu^a} u_{x_i x_j} = f''(u) \,u_{x_i}\,u_{x_j} + f'(u) \, u_{x_i\,x_j}\in C^\alpha(\R^{2m})$ for all indices $i$ and $j$ from $1$ to $2m$, we can apply Lemma 4.5 from \cite{CabreSireI} to obtain that
$$
||\lambda^a \, u_\lambda||_{C^\beta\left(\R^{2m}\times [0,1]\right)} + ||\lambda^a \, (u_{x_i})_\lambda||_{C^\beta\left(\R^{2m}\times [0,1]\right)}  + ||\lambda^a \, (u_{x_i\,x_j})_\lambda||_{C^\beta\left(\R^{2m}\times [0,1]\right)} \!\leq C,
$$
for some $C>0$ and $\beta\in(0,1)$ depending only on $m$, $\s$, $||f||_{C^{2,\alpha}}$, and $||u||_{L^\infty(\R^{2m+1}_+\!)}$.

Now, since $u$ depends only on $s$, $t$ and $\lambda$, from the previous results we obtain
\begin{align*}
u_s \in C^{2,\alpha}(\R_+^{2m+1}\setminus\{s=0,\ \lambda \geq 0\}), \,\, &\,\, \lambda^a\,(u_s)_\lambda \in C^{\alpha}(\overline{\R_+^{2m+1}}\setminus\{s=0,\ \lambda \geq 0\}),\\
u_t \in C^{2,\alpha}(\R_+^{2m+1}\setminus\{t=0,\ \lambda \geq 0\}), \,\, &\,\, \lambda^a\,(u_t)_\lambda \in C^{\alpha}(\overline{\R_+^{2m+1}}\setminus\{t=0,\ \lambda \geq 0\}),\\
u_{st} \in C^{2,\alpha}(\R_+^{2m+1}\setminus\{st=0,\ \lambda \geq 0\}), \,\, &\,\, \lambda^a\,(u_{st})_\lambda \in C^{\alpha}(\overline{\R_+^{2m+1}}\setminus\{st=0,\ \lambda \geq 0\}).
\end{align*}
Furthermore, as it is explained in Section~4 of \cite{Cabre-Saddle}, the regularity and the symmetry of $u$, in $s$ and $t$, yield
\begin{align*}
u_s=0 \,\text{ in } \, \{s=0, \lambda \geq 0\}, \ 
u_t=0 \,\text{ in } \, \{t=0, \lambda \geq 0\},  \
u_{st}=0 \,\text{ in } \, \{st=0, \lambda \geq 0\},
\end{align*}
and
$$ u_s, \ u_t,\ u_{st} \in C(\overline{\R^{2m+1}_+}).  $$

Before proceeding to the proof of Proposition~\ref{Prop:MonotonicityProperties}, we first need the following asymptotic result for the second derivatives in $x$ of $u$. This derivative was not included in the asymptotic theorem of \cite{Cinti-Saddle, Cinti-Saddle2}. We will use it to show that $u_{st}>0$ in $\{s>t>0\}\times [0,+\infty)$.

\begin{lemma}
	\label{Lemma:AsymptoticSecondDerivative}
	Let $f$ satisfy conditions \eqref{Eq:fHypotheses}, and let $u$ be the saddle-shaped solution of \eqref{Eq:AllenCahnWithExtension}. Then, denoting $U(x,\lambda) := u_0((s - t)/\sqrt{2},\lambda) = u_0(z,\lambda)$, we have
	$$ ||D^2_x u(\cdot,\lambda) - D^2_x U(\cdot,\lambda)||_{L^\infty(\R^{2m}\setminus B_R)} \to 0, \ \ \ \text{ as } \ R\to+\infty, $$
	for every $\lambda \in [0,+\infty)$.
\end{lemma}

\begin{proof}
	The proof follows the ones of the analogous results in \cite{Cinti-Saddle2,Cabre-Saddle,CabreTerraII}, where a compactness argument is used. Therefore, we only give here the main ideas, since the details can be found in those papers. Arguing by contradiction, we suppose that the asymptotic result does not hold. Hence, there exists an $\varepsilon>0$ and a sequence $\{x_k\}\subset \ocal$ such that 
	\begin{equation}
	\label{Eq:AsymptoticContradiction}
	|D^2_x u(x_k,\lambda)-D^2_xU(x_k,\lambda)| \geq \varepsilon \ \ \ \text{ and } \ \ \ |x_k|\to+\infty. 
	\end{equation}
	
	Now we distinguish two cases, depending on whether the sequence $\{ \dist(x_k, \ccal) \}$ is unbounded or bounded. In the first case, we show that, up to a subsequence, the function $u_k(x,\lambda) := u(x+x_k,\lambda)$ converges to a solution $u_\infty$ of the semilinear Neumann problem in the half-space $\R^{2m+1}_+$ appearing in the statement of Theorem 5.3 in \cite{Cinti-Saddle2} (see \cite{LiZhang} for the proof). Using this result and the stability of $u_\infty$ we get that $u_\infty \equiv 1$. Thus,  $|D^2_x u(x_k,\lambda)| \to 0$, and since $|D^2_x U(x_k,\lambda)| \to 0$, we arrive at a contradiction with \eqref{Eq:AsymptoticContradiction}. 
	
	In the second case, we have $\dist(x_k, \ccal) = |x_k - x_k^0|$ bounded, where $x_k^0\in \ccal$. Since the Simons cone converges to a hyperplane at infinity (see the details in \cite{CabreTerraII}), it can be proved that, up to a subsequence and a rotation, the function $u_k(x,\lambda) := u(x+x_k^0,\lambda)$ converges to a positive solution $u_\infty$ of an equation in the quarter-space $\R^{2m+1}_{++} = \R^{2m+1}_{+} \cap \{x_{2m}>0\} $ with zero Dirichlet boundary conditions, as in the statement of Theorem 5.5 in \cite{Cinti-Saddle2} (see \cite{Tan} for the proof). Applying this last theorem and the stability again, we conclude that $u_\infty$ must be the 2D solution $u_0$ depending only on $x_{2m}$ and $\lambda$. Hence, $D^2_x(u-U)(x_k,\lambda)$ converges to zero, and we arrive at a contradiction with \eqref{Eq:AsymptoticContradiction}.
\end{proof}

With the help of the maximum principle of Proposition~\ref{Prop:MaxPrincipleLinearizedOperator}, the asymptotic result for the saddle-shaped solution, and the monotonicity properties of the layer solution, we can prove Proposition~\ref{Prop:MonotonicityProperties}.

\begin{proof}[Proof of Proposition~\ref{Prop:MonotonicityProperties}]
	We write \eqref{Eq:AllenCahnWithExtension} in $(s,t,\lambda)$ variables:
	\begin{equation}
	\label{Eq:AllenCahnEithExtensionST}
	\beqc{\PDEsystem}
	\! \! u_{ss} + u_{tt} + u_{\lambda \lambda}  & = & -  (m-1)\bpar{\dfrac{u_s}{s} + \dfrac{u_t}{t}}  - \dfrac{a}{\lambda} u_\lambda  & \text{in } \{st>0,\ \lambda > 0\},\! \! \! \!  \\
	\vspace{1mm}
	u_s & = & 0& \text{on } \{s=0, \ \lambda \geq 0 \},\! \! \! \!  \\
	u_t& = & 0& \text{on } \{t=0, \ \lambda \geq 0 \},\! \! \! \!  \\
	d_\s \dfrac{\partial u}{\partial \nu^a}& = & f(u)& \text{on } \{\lambda = 0\}.\! \! \! \! 
	\eeqc
	\end{equation}
	
	Differentiating the previous equation with respect to $s$ we find that
	$$
	\beqc{\PDEsystem}
	\div (\lambda^a \nabla u_s) &=& (m-1) \dfrac{\lambda^a}{s^2}u_s & \text{ in } \{s>t,\ \lambda > 0\}\,, \\
	d_\s \dfrac{\partial u_s}{\partial \nu^a}& = & f'(u)u_s& \text{ on } \{s>t,\ \lambda = 0\}\,.
	\eeqc
	$$
	Since $u = 0$ on $\{s = t,\ \lambda \geq 0 \}$ and $u>0$ in $\{s > t,\ \lambda \geq 0 \}$, we have that $u_s \geq 0$ on $\partial_L \{s>t,\ \lambda > 0 \} =\{s = t,\ \lambda \geq 0 \}$. Moreover, by the asymptotic result (point (ii) of Theorem~\ref{Th:Summary}) and the monotonicity properties of the layer solution $u_0$ (Lemma~\ref{Lemma:MonotonicityLayer}), we have
	$$
	\liminf_{\{s>t\} ,\ |(s,t)|\to +\infty} u_s(s,t,0) \geq 0\,.
	$$
	Indeed, if $u_0$ is the layer solution, 
	$$
	\partial_s U (x,0) = \dfrac{1}{\sqrt{2}} \partial_x u_0\bpar{\dfrac{s-t}{\sqrt{2}}, 0} \geq 0 
	$$
	and 
	$$
	\lim_{R\to +\infty} \norm{(u_s- \partial_s U)(\cdot,0)}_{L^\infty(\R^{2m}\setminus B_{R})} = 0\,.
	$$
	Thus, by the maximum principle for the linearized operator (Proposition~\ref{Prop:MaxPrincipleLinearizedOperator}) applied to $v=-u_s$, with $b(x,\lambda) = -(m-1)\lambda^a / s^2\leq 0$, we conclude that $u_s \geq 0$ in $\{s \geq t,\ \lambda \geq 0\}$.
	
	Similarly, if we differentiate \eqref{Eq:AllenCahnEithExtensionST} with respect to $t$, we obtain
	$$
	\beqc{\PDEsystem}
	\div (\lambda^a \nabla u_t) &=& (m-1) \dfrac{\lambda^a}{t^2}u_t & \text{ in } \{ s > t > 0 ,\ \lambda > 0\}\,, \\
	d_\s \dfrac{\partial u_t}{\partial \nu^a}& = & f'(u)u_t& \text{ on }  \{ s > t > 0 , \ \lambda = 0\}\,.
	\eeqc
	$$
	In the lateral boundary $\partial_L \{ s > t > 0 ,\ \lambda > 0\} = \{s = t,\ \lambda \geq 0 \}\cup \{t=0, \ \lambda \geq 0\}$ we have  $-u_t \geq 0$. Indeed, $u_t= 0$ on $\{t=0, \ \lambda \geq 0\}$, and since $u = 0$ on $\{s = t,\ \lambda \geq 0 \}$ and $u>0$ in $\{s > t,\ \lambda \geq 0 \}$, it holds $-u_t \geq 0$ on $\{s = t,\ \lambda \geq 0 \}$. Furthermore, the asymptotic behavior of $u$ and the monotonicity properties of the layer solution $u_0$ yield
	$$
	\limsup_{\{s> t>0\} ,\ |(s,t)|\to +\infty}  u_t(s,t,0) \leq 0\,.
	$$
	Indeed, 
	$$
	\partial_t U (x,0) = -\dfrac{1}{\sqrt{2}} \partial_1 u_0\bpar{\dfrac{s-t}{\sqrt{2}}, 0} \leq 0 
	$$
	and 
	$$
	\lim_{R\to +\infty} \norm{(u_t - \partial_t U) (\cdot,0) }_{L^\infty(\R^{2m}\setminus B_{R})} = 0\,.
	$$
	Thus, using again the maximum principle for the linearized operator we find that $-u_t \geq 0$ in $\{ s \geq t,\ \lambda \geq 0\}$.
	
	By the odd symmetry of $u$, i.e., $u(s,t)=-u(t,s)$, we conclude that  $u_s \geq 0$ and $u_t \leq 0$ in $\R^{2m} \times [0, +\infty)$. This fact and the strong maximum principle give that $u_s > 0$ in $(\R^{2m}\setminus \{s=0\}) \times [0, +\infty)$ and $- u_t > 0$ in $(\R^{2m}\setminus \{t=0\}) \times [0, +\infty)$.
	
	Now we check the sign of the $y$-derivative. We use that $\partial_y = (\partial_s + \partial_t)/\sqrt{2}$ to see that
	$$
	\div (\lambda^a \nabla u_y) = (m-1) \dfrac{\lambda^a}{\sqrt{2}}\bpar{ \dfrac{u_s}{s^2} + \dfrac{u_t}{t^2} } = (m-1) \dfrac{\lambda^a}{s^2} u_y + (m-1) \dfrac{\lambda^a}{\sqrt{2}} \dfrac{s^2 - t^2}{s^2t^2} u_t \, . 
	$$
	Hence, using that $u_t \leq 0$ in $ \{ s > t>0,\ \lambda > 0\} $ we get
	$$
	\beqc{\PDEsystem}
	\div (\lambda^a \nabla u_y) &\leq & (m-1) \dfrac{\lambda^a}{s^2} u_y & \text{ in } \{ s > t > 0 ,\ \lambda > 0\}\,, \\
	d_\s \dfrac{\partial u_y}{\partial \nu^a}& = & f'(u)u_y& \text{ on }  \{ s > t > 0 , \ \lambda = 0\}\,.
	\eeqc
	$$
	Note that, since $u$ vanishes at $\ccal \times [0, +\infty)$, $u_y = 0$ on $\{ s=t, \ \lambda \geq 0\}$. Moreover, $u_s\geq 0$ and $u_t = 0$ on $\{t=0, \ \lambda \geq 0\}$. Therefore, $u_y \geq 0$ on $\partial_L \{ s > t > 0 ,\ \lambda > 0\} = \{s = t,\ \lambda \geq 0 \}\cup \{t=0, \ \lambda \geq 0\}$. Furthermore, by the asymptotic behavior of $u$ and the monotonicity properties of the layer solution $u_0$ we have
	$$
	\liminf_{\{s> t>0\} ,\ |(s,t)|\to +\infty}  u_y(s,t,0) = 0\,,
	$$
	since
	$$
	\partial_y U (x,0) = \partial_y u_0(z,0)=  0 \quad \text{ and } \lim_{R\to +\infty} \norm{(u_y - \partial_y U) (\cdot,0) }_{L^\infty(\R^{2m}\setminus B_{R})} = 0\,.
	$$
	Again, by using the maximum principle of Proposition~\ref{Prop:MaxPrincipleLinearizedOperator}, we deduce that $u_y \geq 0$ in $\{ s \geq t,\ \lambda \geq 0\}$, and the strong maximum principle yields $u_y > 0$ on $\{ s > t,\ \lambda \geq 0\}$.
	
	Finally, we prove the last statement concerning the crossed derivatives. By differentiating \eqref{Eq:AllenCahnEithExtensionST}, first with respect to $s$ and then with respect to $t$, we find
	$$
	\ \ 
	\beqc{\PDEsystem}
	\div (\lambda^a \nabla u_{st}) & = & \!(m-1)\lambda^a \bpar{\dfrac{1}{s^2} + \dfrac{1}{t^2} } u_{st} & \text{ in } \{ s > t > 0 ,\ \lambda > 0\}, \\
	d_\s \dfrac{\partial u_{st}}{\partial \nu^a}& = & \!f'(u)u_{st} + f''(u)u_s u_t \geq  f'(u)u_{st}& \text{ on }  \{ s > t > 0 , \ \lambda = 0\}.
	\eeqc
	$$
	Here we have used that $f''(\tau) \leq 0$ if $\tau \in (0,1)$ and that $u_s u_t \leq 0$ in  $\{ s > t > 0 , \ \lambda = 0\}$. Note that, by symmetry, $u_{st}= 0$ on $\{s=t, \ \lambda \geq 0 \}$. Moreover, since $u_t(s,0,\lambda) = 0$ for every $s > 0$ and $\lambda \geq 0$, $u_{st}= 0$ on $\{t=0, \ \lambda \geq 0 \}$. Therefore, $u_{st}= 0$ on $\partial_L \{ s > t > 0 ,\ \lambda > 0\}$. In addition, by the asymptotic result of Lemma~\ref{Lemma:AsymptoticSecondDerivative} and the monotonicity properties of the layer solution $u_0$ (Lemma~\ref{Lemma:MonotonicityLayer}), we have
	$$
	\liminf_{\{s> t>0\} ,\ |(s,t)|\to +\infty}  u_{st}(s,t,0) \geq 0\,,
	$$
	since
	$$
	U_{st} (x,0) = -\dfrac{1}{2} \partial_1^2 u_0\bpar{\dfrac{s-t}{\sqrt{2}}, 0} \geq 0 $$
	and
	$$
	\lim_{R\to +\infty} \norm{(u_{st} - U_{st})(\cdot,0) }_{L^\infty(\R^{2m}\setminus B_{R})} = 0\,.
	$$
	Hence, by the maximum principle for the linearized operator (Proposition~\ref{Prop:MaxPrincipleLinearizedOperator}), we deduce that $u_{st} \geq 0$ in $\{ s \geq t,\ \lambda \geq 0\}$, and the strong maximum principle yields $u_{st} > 0$ in $\{ s > t > 0,\ \lambda \geq 0\}$.
\end{proof}

\section{Stability of the saddle-shaped solution and the Simons cone in dimensions $2m\geq 14$}
\label{Sec:Stability}

In this last section we prove our stability results. The first one is Theorem~\ref{Thm:Stability} and it establishes the stability of the saddle-shaped solution in dimensions $2m\geq 14$. The proof follows the strategy of its analogue in \cite{Cabre-Saddle} and it is based on finding a positive supersolution to the linearized problem. 

\begin{proof}[Proof of Theorem~\ref{Thm:Stability}]
	
	Let us show that $\varphi = t^{-b}u_s - s^{-b}u_t$, with $b(b-m+2)+m-1\leq 0$ and $b>0$, is a positive supersolution of the linearized operator. That is, it satisfies 
	\begin{equation}
	\label{Eq:VarphiSupersolution1}
	\varphi>0 \quad \text{ in } \overline{\R^{2m+1}_+}\setminus \{st=0, \ \lambda >0\}\,,
	\end{equation} 
	\begin{equation}
	\label{Eq:VarphiSupersolution2}
	-\div(\lambda^a \nabla \varphi) \geq 0 \quad \text{ in } \R^{2m+1}_+\setminus \{st=0, \ \lambda >0\}\,,
	\end{equation} 
	and  
	\begin{equation}
	\label{Eq:VarphiSupersolution3}
	\mathscr{L}_u \varphi \geq 0 \quad \text{ on } \R^{2m}  \setminus \{st=0\}\,.
	\end{equation} 
	
	Indeed, note that $\varphi>0$ in $\{s>t>0, \ \lambda \geq0\}$ by the monotonicity properties of $u$ (Proposition~\ref{Prop:MonotonicityProperties}). Since $\varphi$ is even with respect to the Simons cone, i.e., $ \varphi(t,s,\lambda) = \varphi(s,t,\lambda)$, it holds \eqref{Eq:VarphiSupersolution1}. Moreover, \eqref{Eq:VarphiSupersolution3} follows readily, since $\varphi$  satisfies 
	$$ 
	d_\s \frac{\partial \varphi}{\partial \nu^a} = f'(u)\,\varphi\,. 
	$$
	
	Let us now show \eqref{Eq:VarphiSupersolution2}. Since $\varphi$ is even with respect to the Simons cone, it is enough to check that $\div{(\lambda^a\,\nabla\varphi)}\leq 0$ in $\{s>t>0, \ \lambda >0\}$. By using that $\div(\lambda^a\,\nabla u) = 0$, we obtain by a direct computation that 
	\begin{align*}
	\lambda^{-a} \div(\lambda^a\,\nabla \varphi) &= b(b-m+2) \left( t^{-b-2}\,u_s - s^{-b-2}\,u_t \right) \\
	& \hspace{10mm} + (m-1) \left( t^{-b} s^{-2} u_s - s^{-b} t^{-2} u_t \right) \\
	& \hspace{10mm} + 2b \left( t^{-b-1} - s^{-b-1} \right) u_{st}\,.
	\end{align*}
	Now, by using that $u_{st}>0$, $u_y>0$ and $-u_t>0$ in $\{s>t>0, \ \lambda >0\}$, and the fact that $b>0$ satisfies $b(b-m+2)\leq -(m-1) $, we arrive at
	\begin{align*}
	\lambda^{-a} \div(\lambda^a\,\nabla \varphi) &\leq t^{-b} (u_s+u_t) \left( (m-1) s^{-2} + b(b-m+2) t^{-2}\right) \\
	& \hspace{10mm} - t^{-b} u_t \left\{ (m-1) s^{-2} + b(b-m+2) t^{-2} \right\} \\
	& \hspace{10mm} - s^{-b} u_t \left\{ (m-1) t^{-2} + b(b-m+2) s^{-2} \right\} \\
	&= \sqrt{2} t^{-b} u_y \left( (m-1) s^{-2} + b(b-m+2) t^{-2}\right) \\
	& \hspace{10mm} + (-u_t) (m-1) \left( t^{-b} s^{-2} + s^{-b} t^{-2} \right) \\
	& \hspace{10mm} + (-u_t) b (b-m+2) \left( t^{-2-b} + s^{-2-b} \right) \\
	&= \sqrt{2} (m-1) t^{-b} u_y \left( s^{-2} - t^{-2}\right) \\
	& \hspace{10mm} + (-u_t) (m-1) \left( t^{-b} s^{-2} + s^{-b} t^{-2} - t^{-2-b} - s^{-2-b} \right) \\
	& \leq (-u_t) (m-1) (s^{-b}-t^{-b}) (t^{-2}-s^{-2}) \\
	& \leq 0\,.
	\end{align*}
	Note that the existence of $b>0$ such that $b(b-m+2)\leq -(m-1)$ is guaranteed by the assumption $2m\geq 14$.

	Finally, let us show that since we have a positive supersolution to the linearized operator on $\R^{2m}\setminus \{st=0\}$, the stability of $u$ follows. We must check that \eqref{Eq:StabilityCondition} holds. To do it, let us first take nonnegative functions $\zeta\in C^1(\R^{2m+1}_+)$ with compact support in $\{st>0, \ \lambda \geq0\}$. Multiply \eqref{Eq:VarphiSupersolution2} by $\zeta$ and integrate by parts. Using \eqref{Eq:VarphiSupersolution3} we obtain
	\begin{equation}
	\label{Eq:IneqStabilityProof1}	
	\int_{\{st>0\}} f'(u) \, \varphi \, \zeta  \d x \leq d_\s \int_0^\infty \int_{\{st>0\}} \lambda^a \nabla{\varphi} \cdot \nabla{\zeta}  \d x \d \lambda\,.
	\end{equation}
	Now, let $\overline{\xi} \in C^\infty_c(\overline{\R^{2m+1}_+}\setminus \{st=0, \ \lambda \geq 0\})$. Since $\varphi > 0$ in $\{st>0, \lambda \geq 0\}$, taking $\zeta = \overline{\xi}^2/\varphi$ in \eqref{Eq:IneqStabilityProof1} and using the Cauchy-Schwarz inequality, we get
	\begin{align*}
	\int_{\{st>0\}}  \! f'(u)\,\overline{\xi}^2 \d x &= \int_{\{st>0\}} \! f'(u)\,\varphi \frac{\overline{\xi}^2}{\varphi} \d x \leq d_\s \int_0^\infty \! \!  \int_{\{st>0\}} \! \lambda^a \nabla{\varphi} \cdot \nabla{\left(\frac{\overline{\xi}^2}{\varphi}\right)} \d x \d \lambda \\
	&= d_\s \int_0^\infty \! \! \int_{\{st>0\}}\!  \lambda^a\,\frac{2\overline{\xi}}{\varphi} \,\nabla \varphi\cdot \nabla \overline{\xi}  \d x \d \lambda \\
	& \quad \ \quad - d_\s \int_0^\infty \! \!  \int_{\{st>0\}} \!  \lambda^a\,\frac{\overline{\xi}^2}{\varphi^2}|\nabla \varphi|^2  \d x \d \lambda \\
	&\leq d_\s \int_0^\infty \! \! \int_{\{st>0\}} \! \lambda^a\,|\nabla \overline{\xi}|^2  \d x \d \lambda\,.
	\end{align*}

	To conclude the proof, let us show that the last inequality holds for every smooth function $\xi$ with compact support in $\overline{\R^{2m+1}_+}$. This will yield the stability of $u$. Take $\eta_\varepsilon\in C^\infty(\R)$ such that $0\leq \eta_\varepsilon \leq 1$ and
	$$ \eta_\varepsilon = \begin{cases}
	1 \ \ \ \text{ in } \ \ \ [\varepsilon, +\infty)\,,\\
	0 \ \ \ \text{ in } \ \ \ [0, \varepsilon/2)\,.\\
	\end{cases} $$
	Then, since $\xi\,\eta_\varepsilon(s)\,\eta_\varepsilon(t)$ has compact support in $\{st>0, \ \lambda \geq 0\}$, we can replace $\overline{\xi}$ by $\xi\,\eta_\varepsilon(s)\,\eta_\varepsilon(t)$ in the previous inequality to get
	$$
	\frac{1}{d_\s}\int_{\R^{2m}} f'(u)\,\xi^2\,\eta_\varepsilon^2(s)\,\eta_\varepsilon^2(t) \d x \leq \int_{\R^{2m+1}_+} \lambda^a\,|\nabla (\xi\,\eta_\varepsilon(s)\,\eta_\varepsilon(t))|^2 \d x \d \lambda\,. $$
	Now, we compute the terms in the right-hand side of this inequality. By using Cauchy-Schwarz, we see that to deduce the stability condition 
	$$
	\frac{1}{d_\s}\int_{\R^{2m}} f'(u)\,\xi^2 \d x  
	\leq \int_{\R^{2m+1}_+} \lambda^a |\nabla \xi|^2  \d x \d \lambda
	$$
	by letting $\varepsilon \to 0$, it is enough to show that
	$$ \int_{\R^{2m+1}_+} \lambda^a|\nabla\eta_\varepsilon(s)|^2 \d x \d \lambda  \to 0 \ \ \ \text{ as } \ \ \varepsilon \to 0\,,$$
	and the same with $\eta_\varepsilon(s)$ replaced by $\eta_\varepsilon(t)$. To see this, let $R>0$ be such that $\text{supp}(\xi)\subset \overline{B_R^+}$. Then, since $m\geq 3$,
	\begin{align*}
	\int_{\R^{2m+1}_+} \lambda^a|\nabla\eta_\varepsilon(s)|^2 \d x \d \lambda &\leq \frac{C}{\varepsilon^2} \int_0^R \d \lambda \, \lambda^a \int_0^\varepsilon \d s \, s^{m-1} \int_0^R \d t \,t^{m-1} \\
	&\leq C\,R^{m+a+1}\,\varepsilon^{m-2} \to 0 \ \ \ \text{ as } \ \ \varepsilon \to 0\,,
	\end{align*}
	The computation is analogous for $\eta_\varepsilon(t)$.
\end{proof}

Finally, we present the proof of the stability of the Simons cone as a nonlocal $(2\s)$-minimal surface whenever $2m\geq 14$ and $\s\in(0,1/2)$.

\begin{proof}[Proof of Corollary~\ref{Cor:SimonsConeStableDim14}]
	Let $u$ be the saddle-shaped solution of \eqref{Eq:AllenCahn} in dimension $2m\geq 14$. Consider the blow-down sequence $u_k(x) = u(kx)$ with $k\in \N$. On the one hand, since $u$ is stable in such dimensions and $\s\in(0,1/2)$, by Theorem~2.6 in \cite{CabreCintiSerra-Stable} there exists a subsequence $k_j$ such that
	$$ u_{k_j} \to \chi_{\Sigma}-\chi_{\R^{2m}\setminus\Sigma} \quad \textrm{ in } L^1(B_1) \quad \textrm{ as } k_j\to+\infty\,, $$
	for some cone $\Sigma$ that is a stable set for the fractional perimeter.
	
	On the other hand, by the asymptotic behavior of $u$ (point (ii) in Theorem~\ref{Th:Summary}) it is clear that
	$$ u_{k} \to \chi_{\ocal}-\chi_{\R^{2m}\setminus\ocal} \quad \textrm{ a.e. as } k\to+\infty\,.$$
	
	Putting all together we conclude that $\ocal$ is a stable set for the fractional perimeter if $2m\geq 14$ and $\s\in(0,1/2)$. This is the same as saying that the Simons cone is a stable nonlocal $(2\s)$-minimal surface in such dimensions.
\end{proof}

\section*{Acknowledgements}

The authors thank Xavier Cabr\'e for his guidance and useful discussions on the topic of this paper.

\bibliographystyle{amsplain}
\bibliography{biblio}

\end{document}

%% file: Examples.tex
\definecolor{CustomGreen}{RGB}{118,230,165}

\begin{tikzpicture}[y=0.80pt, x=0.80pt, yscale=-0.4100000, xscale=0.4100000, inner sep=0pt, outer sep=0pt]



\path[left color=CustomGreen, right color=transparent!0, path fading=north, opacity = 0.5, line width=0.424pt] (118.2429,764.3301) -- (21.1789,715.5369)
.. controls (94.5967,615.3536) and (195.9205,579.7926) .. (247.1606,579.7926)
.. controls (298.4006,579.7926) and (398.5544,619.4041) .. (472.6884,717.0446)
-- (285.9016,809.4614) .. controls (233.3585,809.4950) and (268.1303,807.9929)
.. (207.9356,807.9929) -- cycle;

\path[color=black,fill=black,line join=miter,line cap=butt,miter limit=4.00,even
odd rule,line width=3.168pt] (18.5176,807.9742) -- (18.5176,809.8759) --
(475.5404,809.8759) -- (475.5404,807.9742) -- cycle;

\path[draw=black,fill=black,even odd rule,line width=1.475pt]
(474.1177,809.0013) -- (467.2064,815.9192) -- (491.3962,809.0013) --
(467.2064,802.0833) -- cycle;

\path[color=black,fill=black,line join=miter,line cap=butt,miter limit=4.00,even
odd rule,line width=2.404pt] (246.0612,807.5509) -- (247.9610,807.5509) --
(247.9610,544.0301) -- (246.0612,544.0301) -- cycle;

\path[draw=black,fill=black,even odd rule,line width=1.475pt]
(247.0622,550.5293) -- (253.9736,557.4472) -- (247.0622,533.2344) --
(240.1508,557.4472) -- cycle;

\path[color=black,fill=black,line join=miter,line cap=butt,miter limit=4.00,even
odd rule,line width=1.786pt] (207.5049,806.5862) -- (207.5049,811.5256) --
(286.2549,811.5256) -- (286.2549,806.5862) -- cycle;

\path[draw=black,line join=miter,line cap=butt,even odd rule,line width=1.492pt]
(207.3807,799.1032) -- (207.3807,818.6875);

\path[draw=black,line join=miter,line cap=butt,even odd rule,line width=1.587pt]
(286.1082,799.1032) -- (286.1082,818.6875);

\path[draw=black,line join=miter,line cap=butt,miter limit=4.00,line
width=0.6pt] (207.3790,808.5203) -- (20.6223,716.0642);

\path[draw=black,line join=miter,line cap=butt,miter limit=4.00,line
width=0.6pt] (285.9016,809.4614) -- (472.6884,717.0446);

\node at (320, 670) {\normalsize $\Omega_1$};
\node at (200, 840) {\normalsize $-\varepsilon$};
\node at (287, 840) {\normalsize $\varepsilon$};
\node at (248, 832) {\normalsize $\Gamma_1$};
\node at (525, 810) {\normalsize $\R$};
\node at (247, 500) {\normalsize $\lambda>0$};


\path[left color=CustomGreen, right color=transparent!0, path fading=west, opacity = 0.9, line width=0.849pt]
(1010,770) -- (580,770) -- (580,810) --(1010,810) -- cycle;

\path[draw=black,fill=black,even odd rule,line width=1.475pt]
(1004.5579+30,809.5433) -- (997.6465+30,816.4612) -- (1021.8363+30,809.5433) --
(997.6465+30,802.6253) -- cycle;

\path[draw=black,fill=black,even odd rule,line width=1.475pt]
(777.5024+30,551.0713) -- (784.4138+30,557.9892) -- (777.5024+30,533.7765) --
(770.5911+30,557.9892) -- cycle;

\path[color=black,fill=black,line join=miter,line cap=butt,miter limit=4.00,even
odd rule,line width=2.404pt] (776.5014+30,809.5942) -- (778.4012+30,809.5942) --
(778.4012+30,546.0734) -- (776.5014+30,546.0734) -- cycle;

\path[color=black,fill=black,line join=miter,line cap=butt,miter limit=4.00,even
odd rule,line width=3.168pt] (550.2910+30,769.4678) -- (550.2910+30,771.3694) --
(1007.3138+30,771.3694) -- (1007.3138+30,769.4678) -- cycle;

\path[color=black,fill=black,line join=miter,line cap=butt,miter limit=4.00,even
odd rule,line width=4.303pt] (548.9578+30,806.6413) -- (548.9578+30,811.5806) --
(1006.1195+30,811.5806) -- (1006.1195+30,806.6413) -- cycle;

\path[color=black,fill=black,line join=miter,line cap=butt,miter limit=4.00,even
odd rule,line width=2.181pt] (767.3738+30,768.3177) -- (767.3738+30,772.3021) --
(786.9597+30,772.3021) -- (786.9597+30,768.3177) -- cycle;

\node at (900+30, 790) {\normalsize $\Omega_2$};
\node at (805+30, 755) {\normalsize $\varepsilon$};
\node at (700+50, 832) {\normalsize $\Gamma_2$};
\node at (1050+30, 810) {\normalsize $\R$};
\node at (773+30, 500) {\normalsize $\lambda>0$};

\end{tikzpicture}

%% file: ExtensionNarrow.tex
\definecolor{CustomGreen}{RGB}{118,230,165}

\begin{tikzpicture}[y=0.80pt, x=0.80pt, yscale=-0.6, xscale=0.6, inner sep=0pt, outer sep=0pt]
    
  	\path[scale=0.938, left color=CustomGreen, right color=transparent!0, path fading=south, fading transform={rotate=165}, opacity = 0.5, line width=0.566pt] 
   	(184.5972,730.9394) ..
    controls (180.9575,730.9394) and (175.7847,731.1061) .. (175.6080,730.5759) ..
    controls (171.9672,733.0400) and (176.5782,725.4990) .. (177.0743,722.3815) ..
    controls (177.9067,717.1499) and (178.9183,706.8342) .. (179.0403,699.1456) ..
    controls (179.1420,692.7318) and (179.4706,682.7353) .. (179.7596,674.0350) ..
    controls (180.8080,642.4728) and (181.7457,635.5373) .. (188.5866,610.8525) ..
    controls (194.0462,591.1524) and (196.2268,586.3680) .. (197.4590,578.2401) ..
    controls (198.7299,569.8573) and (197.9960,549.5841) .. (196.8925,533.9029) ..
    controls (196.1731,523.6790) and (196.0775,513.8122) .. (196.1359,504.1440) ..
    controls (196.1807,496.7341) and (198.3317,489.1993) .. (199.2864,483.2319) ..
    controls (200.3263,476.7322) and (207.0868,463.2405) .. (207.9281,461.6518) ..
    controls (209.5004,458.6828) and (212.8173,454.6922) .. (215.4496,451.0385) ..
    controls (219.4043,445.5493) and (224.1839,440.4696) .. (226.8947,437.6896) ..
    controls (230.7942,433.6906) and (235.0627,430.1254) .. (239.2315,426.3742) ..
    controls (245.7129,420.5421) and (252.4112,415.8531) .. (258.2915,411.7250) ..
    controls (262.4443,408.8096) and (275.7732,400.1228) .. (289.5695,388.7998) ..
    controls (304.1982,376.7937) and (320.8984,362.8705) .. (332.5668,350.4120) ..
    controls (342.1077,340.2249) and (352.7447,329.9966) .. (360.9981,320.1828) ..
    controls (366.8928,313.1736) and (372.7162,306.7946) .. (376.9603,300.9333) ..
    controls (380.9052,295.4854) and (385.1250,290.9400) .. (387.8389,287.3902) ..
    controls (394.4930,278.6863) and (396.4141,275.0680) .. (396.5070,275.1820) ..
    controls (396.6358,275.3403) and (414.7734,312.7490) .. (439.3322,364.6649) ..
    controls (486.7650,464.9352) and (559.0040,618.9806) .. (577.6929,660.8923) ..
    controls (580.6800,667.5911) and (582.5384,671.3050) .. (582.4448,671.6076) ..
    controls (581.1122,675.9196) and (566.5199,682.2313) .. (562.4424,684.3159) ..
    controls (557.7739,686.7027) and (533.7931,692.6797) .. (523.4964,694.4520) ..
    controls (517.9932,695.3993) and (491.2056,692.2455) .. (474.5341,686.9823) ..
    controls (446.4561,678.1181) and (419.6258,662.6193) .. (401.5646,675.2542) ..
    controls (388.1532,684.6362) and (393.3733,680.8042) .. (390.9447,688.7330) ..
    controls (387.6971,699.3350) and (399.0009,708.5054) .. (405.6164,720.1689) ..
    controls (405.6164,720.1689) and (407.3428,724.2187) .. (407.3428,731.3123) ..
    controls (385.1788,731.3123) and (265.2897,730.8248) .. (246.4143,730.8248) ..
    controls (225.9070,729.9416) and (204.6809,731.0053) .. (184.5972,730.9394) --
    cycle;
    
  	\path[draw=black,line join=miter,line cap=round,miter limit=4.00,even odd
    rule,line width=2.732pt] (164.6325,684.9149) .. controls (165.4599,680.7181)
    and (166.1341,676.4910) .. (166.6533,672.2450) .. controls (169.9255,645.4862)
    and (167.0382,618.0907) .. (172.2994,591.6509) .. controls (175.7296,574.4130)
    and (182.5705,557.9076) .. (184.8401,540.4789) .. controls (186.7951,525.4666)
    and (185.2973,510.2530) .. (184.4387,495.1384) .. controls (183.5802,480.0237)
    and (183.4049,464.5660) .. (187.8713,450.1008) .. controls (193.1075,433.1422)
    and (204.4385,418.6308) .. (217.4800,406.5922) .. controls (230.5216,394.5535)
    and (245.3615,384.6665) .. (259.6084,374.0818) .. controls (303.0951,341.7734)
    and (341.4415,302.5612) .. (372.7712,258.3641);

  	\path[draw=black,line join=miter,line cap=round,miter limit=4.00,even odd
    rule,line width=2.480pt] (381.9422,683.7850) .. controls (382.1756,676.5966)
    and (378.6287,668.3766) .. (372.5114,662.9285) .. controls (371.8317,662.3232)
    and (371.1223,661.7504) .. (370.4734,661.1123) .. controls (368.2437,658.9198)
    and (366.7904,655.9781) .. (366.2816,652.8926) .. controls (365.7728,649.8071)
    and (366.1931,646.5911) .. (367.3626,643.6909) .. controls (368.5320,640.7906)
    and (370.4368,638.2058) .. (372.7785,636.1333) .. controls (375.1202,634.0608)
    and (377.8919,632.4946) .. (380.8346,631.4366) .. controls (386.7201,629.3206)
    and (393.1862,629.2506) .. (399.3581,630.2625) .. controls (405.5300,631.2745)
    and (411.4767,633.3260) .. (417.3824,635.3850) .. controls (440.3746,643.4010)
    and (463.9431,651.7780) .. (488.2825,651.0739) .. controls (509.3214,650.4653)
    and (530.0955,642.8046) .. (546.4925,629.6082);

    \path[draw=black,line join=miter,line cap=butt,miter limit=4.00,even odd rule,line width=1.7pt] (32.3523,685.2737) -- (623.5039,685.2737);  
    \path[draw=black,line join=miter,line cap=butt,miter limit=4.00,even odd rule,line width=1.7pt] (162.5371,225.4824) -- (162.5371,686.4492);
    
    \path[draw=black,fill=black,even odd rule,line width=1.7pt]
    (611.9033,685.3686) -- (608.2633,690.0086) -- (623.5033,685.3686) --
    (608.2633,680.7286) -- cycle;
    \path[draw=black,fill=black,even odd rule,line width=1.7pt]
    (162.5371,237.0830) -- (167.1771,240.7230) -- (162.5371,225.4830) --
    (157.8971,240.7230) -- cycle;

	\path[color=black,fill=black,line join=miter,line cap=butt,miter limit=4.00,even
	odd rule,line width=0.808pt] (157.2344,688.6523) .. controls
	(152.0012,691.5920) and (146.7226,694.4515) .. (141.4004,697.2266) .. controls
	(136.5055,699.7788) and (131.5748,702.2572) .. (126.4297,704.1660) .. controls
	(122.3383,705.6839) and (118.1124,706.8419) .. (113.8496,707.8594) .. controls
	(98.0654,711.6268) and (81.8228,713.4748) .. (65.5957,713.3516) --
	(65.5879,714.3633) .. controls (81.8969,714.4871) and (98.2202,712.6282) ..
	(114.0840,708.8418) .. controls (118.3683,707.8192) and (122.6341,706.6519) ..
	(126.7813,705.1133) .. controls (131.9965,703.1784) and (136.9637,700.6778) ..
	(141.8672,698.1211) .. controls (147.1988,695.3411) and (152.4882,692.4780) ..
	(157.7305,689.5332) -- cycle;
	
	\path[draw=black,fill=black,even odd rule,line width=0.862pt]
	(148.6738,694.0419) -- (147.1296,699.5447) -- (157.4825,689.0937) --
	(143.1710,692.4977) -- cycle;

  	\path[draw=black,line join=miter,line cap=butt,even odd rule,line width=2.056pt]
    (163.8327,727.5372) -- (382.4599,727.5372);

  	\path[draw=black,line join=miter,line cap=butt,even odd rule,line width=1.681pt]
    (163.0954,717.6006) -- (163.0954,737.4660);
    
  	\path[draw=black,line join=miter,line cap=butt,even odd rule,line width=1.788pt]
    (381.9422,717.6006) -- (381.9422,737.4660);
    
  	\path[draw=black,line join=miter,line cap=butt,miter limit=4.00,even odd
    rule,line width=2.040pt] (164.1431,679.2655) -- (164.1431,692.4511);
    
    \path[draw=black,line join=miter,line cap=butt,miter limit=4.00,even odd
    rule,line width=2.040pt] (213.2179,679.2655) -- (213.2179,692.4511);
    
  	\path[draw=black,line join=miter,line cap=butt,miter limit=4.00,even odd
    rule,line width=3.550pt] (164.1431,685.1777) -- (213.2179,685.1777);
    
    \path[fill=black,line cap=round,miter limit=4.00,even odd rule,line
    width=0.000pt]
    (170.0039,686.9028)arc(22.850:108.383:6.275)arc(108.383:193.916:6.275)arc(193.916:279.448:6.275)arc(-80.552:4.981:6.275);

    \node at (160,200) {\normalsize $\lambda>0$};
    \node at (662, 685) {\normalsize $\mathbb R^{2m}$};
    \node at (240, 305) {\normalsize $\ocal \times (0,+\infty)$};
    \node at (80, 305) {\normalsize $\ocal^c \times (0,+\infty)$};
    \node at (390, 455) {\normalsize $\Omega$};
    \node at (52, 712) {\normalsize $\ccal$};
    \node at (273, 748) {\normalsize $\partial_0 \Omega$};    
    \node at (190, 705) {\normalsize $\Gamma$}; 
\end{tikzpicture}